\documentclass[10pt]{article}
\usepackage{color}
\usepackage{amssymb}
\usepackage{amsthm,array,amssymb,amscd,amsfonts,latexsym, url}
\usepackage{amsmath}
\usepackage[all]{xy}
\newtheorem{theo}{Theorem}[section]
\newtheorem{prop}[theo]{Proposition}
\newtheorem{claim}[theo]{Claim}
\newtheorem{lemm}[theo]{Lemma}
\newtheorem{coro}[theo]{Corollary}
\newtheorem{rema}[theo]{Remark}
\newtheorem{Defi}[theo]{Definition}

\title{Schiffer variations and the generic Torelli theorem  for  hypersurfaces}
\author{Claire Voisin\footnote{The author is supported by the ERC Synergy Grant HyperK (Grant agreement No. 854361).}}

\date{}

\newfont{\gothic}{eufb10}
\begin{document}
\maketitle
\setcounter{section}{-1}

\begin{abstract} We prove the generic Torelli theorem for  hypersurfaces in $\mathbb{P}^n$ of degree $d$  dividing  $n+1$, for $d$ sufficiently large. Our proof involves the higher order study of the  variation of Hodge structure along  particular $1$-parameter families of hypersurfaces that we call ``Schiffer variations".
We also analyze the  case of degree $4$.
Combined with Donagi's  generic Torelli theorem and results of Cox-Green, this shows that the generic Torelli theorem for  hypersurfaces
 holds with  finitely many exceptions.
 \end{abstract}

\section{Introduction}
We will consider in this paper  smooth hypersurfaces   $X_f\subset \mathbb{P}^n$   of degree $d$ defined by a
homogeneous  polynomial equation $f$. By the Lefschetz theorem on hyperplane sections, only the degree $n-1$   cohomology group $H^{n-1}(X_f,\mathbb{Z})$
carries a nontrivial Hodge structure, and its primitive part
$H^{n-1}(X_f,\mathbb{Z})_{\rm prim}:={\rm Ker}\,(H^{n-1}(X_f,\mathbb{Z})\rightarrow H^{n+1}(\mathbb{P}^n,\mathbb{Z}))$  carries a Hodge structure polarized by the cup-product $\langle\,\,,\,\,\rangle_X$ (the two groups agree when $n-1$ is odd, otherwise they differ by  $\mathbb{Z}h^{\frac{n-1}{2}}$, where $h=c_1(\mathcal{O}_X(1))$). The global  Torelli problem for hypersurfaces thus asks
whether the existence of an isomorphism of polarized Hodge structures

$$ H^{n-1}(X_f,\mathbb{Z})_{\rm prim}\cong H^{n-1}(X_{f'},\mathbb{Z})_{\rm prim}$$
(extending to an isomorphism of Hodge structures  $H^{n-1}(X_f,\mathbb{Z})\cong H^{n-1}(X_{f'},\mathbb{Z})$ preserving  the classes $h^{\frac{n-1}{2}}$ on both sides when $n-1$ is even)
implies that $X_f\cong X_{f'}$. There are very few cases where this statement is known: for plane curves, we can apply the Torelli theorem for curves. For quartic surfaces, the global Torelli theorem is proved by  Piateski-Shapiro-Shafarevich  \cite{Pysha}. For cubic threefolds, the global Torelli theorem is proved by Clemens-Griffiths \cite{Clegri} and Beauville \cite{beauville}, and for cubic fourfolds it was first  proved in \cite{voisincubic} (alternative  proofs are now available, see e.g. \cite{huybrechts}).

The generic Torelli theorem for hypersurfaces of degree $d$ and dimension $n-1$     is the following  statement  that we will study in this paper:

\vspace{0.5cm}

{\it Let $X_f$ be
a very general smooth hypersurface of degree $d$ in $\mathbb{P}^n$. Then any smooth hypersurface $X_{f'}$ of degree $d$ in $\mathbb{P}^n$ such that there exists an isomorphism of   Hodge structures
\begin{eqnarray}
\label{eqisopol} H^{n-1}(X_f,\mathbb{Q})_{\rm prim}\cong  H^{n-1}(X_{f'},\mathbb{Q})_{\rm prim}
\end{eqnarray}
is isomorphic to $X_f$.}

\vspace{0.5cm}

We will explain in Section \ref{secnounou} why  the ``very general'' assumption is natural in this statement. This is related to   the Cattani-Deligne-Kaplan theorem \cite{CDK} which implies that  the
set of pairs $(f,\,{f'})$ such that an isomorphism as in (\ref{eqisopol}) exists  is a countable union of closed algebraic subsets in $U_{d,n}\times U_{d,n}$, where $U_{d,n}$ is the moduli space of smooth hypersurfaces of degree $d$ in $\mathbb{P}^{n}$.

\begin{rema}{\rm The Torelli theorem is usually    stated for the isomorphisms of polarized Hodge structures. However, using  infinitesimal arguments (see \cite[6.3.1]{voisinbook})  we can see that, except in the  case of cubic surfaces, the Hodge structure on the primitive cohomology of $X_f$ has  a unique polarization (up to a scalar) for very general $X_f$. So the polarized and nonpolarized statements are equivalent.
}
\end{rema}

In the case of cubic surfaces,
the generic Torelli theorem is clearly wrong, since they have moduli, while their variation of Hodge structure is trivial. The case of plane quartics is also a counterexample to the generic Torelli theorem with rational coefficients, since in genus $3$, a  general curve is  not determined by the isogeny class of its Jacobian.  Donagi proved in \cite{donagi}
 the following beautiful result.
\begin{theo} \label{theodonagi} The generic Torelli theorem holds
for smooth hypersurfaces of degree $d$ in $\mathbb{P}^n$, with $n\geq 3$, $(d,n)\not=(3,3)$  and the following possible  exceptions:
\begin{enumerate} \item \label{i} $d$ divides $n+1$,
\item\label{ii} $d=4$, $n=4m+1$, with $m\geq1$,
\item\label{iii} $d=6$, $n=6m+2$, with $m\geq1$.
\end{enumerate}
\end{theo}
\begin{rema}{\rm The Torelli theorem is usually stated for integral Hodge structures, and Donagi's original statement indeed concerned integral Hodge structures. In fact, his proof  works as well  for  rational Hodge structures,
since it relies on the study of the (complex!) variation of Hodge structure for hypersurfaces of given degree and dimension  and its local invariants. Another instance where a generic  Torelli theorem has  been proved for rational Hodge structures is the case of curves of genus $g\geq 4$ which is treated in \cite{bardelli-pirola}. In this case, Bardelli and Pirola prove  that a very general curve of genus at least $4$ is determined by the isogeny class of  its Jacobian. }
\end{rema}
\begin{rema} {\rm The Cattani-Deligne-Kaplan algebraicity theorem mentioned above appeared  much later than  \cite{donagi}, so that Theorem \ref{theodonagi} is in fact a  slightly strengthened    version of Donagi's theorem, taking into account \cite{CDK}.}
\end{rema}
 Cox and Green  solved in \cite{coxgreen} the case \ref{iii}, that is $d=6$, but the two infinite series \ref{i} and \ref{ii} essentially remained open.
The starting point of Donagi's proof is the description due to Griffiths and Carlson-Griffiths of the infinitesimal variation of Hodge
structure of a smooth hypersurface.
 Denote by $S^* =\mathbb{C}[X_0,\ldots,X_n]$ the graded polynomial ring of   $ \mathbb{P}^n$ and by $R_f^*=S^*/J_f^*$ the Jacobian ring of $f$, where
\begin{eqnarray}\label{eqJf}
J_f^* = S^{*-d+1}\langle\frac{ \partial f}{\partial X_i}\rangle\subset S^*\end{eqnarray}
is the Jacobian ideal of $f$, generated by the partial derivatives of $f$.
The infinitesimal variation of Hodge structure on the primitive cohomology of degree $n-1$ of $X_f$
is given, according to Griffiths \cite{gri}, see also \cite[6.1.3]{voisinbook}, by  linear maps
\begin{eqnarray}\label{eqivhs}
R^d_f\rightarrow {\rm Hom}\,(H^{p,q}(X_f)_{\rm prim},H^{p-1,q+1}(X_f)_{\rm prim})
\end{eqnarray}
for $p+q=n-1$.
Here, the space $R^d_f$ is naturally identified with the first order deformations of $X_f$ in
$\mathbb{P}^n$ modulo the infinitesimal  action of  ${\rm PGL}(n+1)$. It also identifies via the Kodaira-Spencer map
to the subspace $H^1(X_f,T_{X_f})_0\subset H^1(X_f,T_{X_f})$ of deformations of $X_f$ induced by  a deformation of $f$.
Griffiths constructs residue isomorphisms
\begin{eqnarray}\label{eqresidue} {\rm Res}_{X_f}:R^{(q+1)d-n-1}_f\stackrel{\cong}{\rightarrow} H^{n-q-1,q}(X_f)_{\rm prim}
\end{eqnarray}
and
the paper \cite{cagri} in turn describes (\ref{eqivhs}) using the isomorphisms  (\ref{eqresidue})  as follows:
\begin{theo} \label{theCagri} Via the isomorphisms (\ref{eqresidue}), the maps
(\ref{eqivhs}) identify up to a scalar coefficient with the map
\begin{eqnarray}\label{eqmulti} R_f^d\rightarrow {\rm Hom}(R^{(q+1)d-n-1}_f,R^{(q+2)d-n-1}_f).
\end{eqnarray} induced by multiplication in $R_f^*$. In other words, the following diagram is commutative up to a coefficient
\begin{eqnarray}\label{numerodiagcagri}
 \xymatrix{
&R^d_f\ar[r]\ar[d]^{\cong}& {\rm Hom}\,(R^{(q+1)d-n-1}_f,R^{(q+2)d-n-1}_f)\ar[d]^{\cong}\\
&H^1(X_f,T_{X_f})_0\ar[r]& {\rm Hom}\,(H^{n-1-q,q}(X_{f})_{\rm prim},H^{n-2-q,q+1}(X_{f})_{\rm prim})
.}
\end{eqnarray}

Furthermore, the Serre  pairing between $H^{n-1-q,q}(X_f)_{\rm prim}$ and $H^{q,n-1-q}(X_f)_{\rm prim} $ identifies with the Macaulay pairing
\begin{eqnarray}\label{eqmacpair} R^{(q+1)d-n-1}_f\otimes R^{(n-q)d-n-1}_f\rightarrow R_f^{(n+1)(d-2)}\cong \mathbb{C}.
\end{eqnarray}
given by the product in $R_f^*$.
\end{theo}

 Donagi's proof  starts with the observation that
Theorem \ref{theodonagi}  is implied by the following result:
\begin{theo} \label{theodonagireform}   Let $X$ be
a  smooth hypersurface of degree $d$ in $\mathbb{P}^n$, with $n\geq 3$. Assume $(d,n)\not=(3,3)$  and we are
 not in the cases \ref{i}, \ref{ii}, \ref{iii} listed in Theorem \ref{theodonagi}.
 Then  $X$ is determined by the  data (\ref{eqmulti}) for all $q$ and the Macaulay pairings (\ref{eqmacpair}), hence, using Theorem \ref{theCagri},  by   its polarized  infinitesimal variation of Hodge
structure.
\end{theo}
Concretely, Theorem \ref{theodonagireform} says that if
$X_f$  and $X_{f'}$ are two smooth hypersurfaces of degree $d$ and dimension $n-1$  such that there exist isomorphisms
$$R_f^d\cong R_{f'}^d,\,\,R^{(q+1)d-n-1}_f\cong R^{(q+1)d-n-1}_{f'}\,\,{\rm for\,\, any}\,\, q, $$
compatible with the Macaulay pairing (\ref{eqmacpair}) for $f$ and $f'$,
 and  such that the following diagram
commutes:
\begin{eqnarray}\label{numerodiagavecR}
\xymatrix{
&R^d_f\ar[r]\ar[d]& \bigoplus_q{\rm Hom}\,(R^{(q+1)d-n-1}_f,R^{(q+2)d-n-1}_f)\ar[d]\\
&R^d_{f'}\ar[r]& \bigoplus_q{\rm Hom}\,(R^{(q+1)d-n-1}_{f'},R^{(q+2)d-n-1}_{f'})
,}
\end{eqnarray}

then $X_f$ is isomorphic to $X_{f'}$.

Donagi's proof of Theorem \ref{theodonagireform} consists in recovering from the data
  (\ref{eqmulti}) its  {\it polynomial structure} (see Section \ref{secpolystruc}), and more precisely, reconstructing the whole Jacobian ring of $f$ from its partial data appearing in (\ref{eqmulti}). He then applies the Mather-Yau theorem (see Proposition  \ref{theoyau}) which  says that $f$ is determined by $J_f^{d-1}\subset S^{d-1}$.

  Donagi's  method  does not work in  the case where $(d,n)=(4,3)$, that is,  quartic $K3$ surfaces because   Theorem \ref{theodonagireform} is clearly wrong  in this case. In fact,  Theorem \ref{theodonagi} is also wrong for quartic surfaces, due to the fact that it is stated for  rational Hodge structures.  More generally, Donagi's  method to recover the polynomial structure, based on the use of the symmetrizer lemma (Proposition  \ref{propsymm}), gives nothing more, when $d$ divides $n+1$, than the subring
$R_f^{*d}\subset R^{*}_f$  defined as the sum of the graded pieces of $R_f^{*}$ of degree divisible by $d$. This is why Donagi's method fails to give the result in that case. The goal of this paper is to extend  Theorem \ref{theodonagi} to  most families  of hypersurfaces
not covered   by Donagi's theorem.
\begin{theo} \label{theodonagivoisin} (1)  The generic Torelli theorem holds
for smooth hypersurfaces of degree $d$ in $\mathbb{P}^n$ when  $d$ divides $n+1$  and $d$ is large enough.
In particular, it holds for Calabi-Yau hypersurfaces of degree $d$ large enough.

(2) The generic Torelli theorem holds
for smooth hypersurfaces of degree $4$  in $\mathbb{P}^{4m+1}$  for $m$ sufficiently large.
\end{theo}

These results combined with Donagi's theorem (Theorem \ref{theodonagi}) and Cox-Green's result in \cite{coxgreen}  imply the following result:
\begin{coro} The generic Torelli theorem holds for   hypersurfaces of degree $d$ in
$\mathbb{P}^n$ with  finitely many exceptions.
\end{coro}
The proof of Theorem \ref{theodonagivoisin} (2) will be given in Section \ref{sec46}.  We will give there an effective estimate for $m$, which can probably be improved by refining the method. In that case, the method of proof follows closely Donagi's ideas, and in particular
passes through a proof of Theorem \ref{theodonagireform},  at least for $X$ generic.

The  case (1) of Theorem \ref{theodonagivoisin} had been also proved in \cite{voisin} in the case of quintic threefolds, the first case which is not covered
by Theorem \ref{theodonagi}, by extending Theorem \ref{theodonagireform} to that case. It is quite possible that Theorem \ref{theodonagireform} is true more generally when $d$ divides $n+1$ and $d$ is  sufficiently large, but the proof given in \cite{voisin} is  very technical and  ad hoc, hence is not encouraging.

Our   proof of Theorem \ref{theodonagivoisin} (1)  also rests on the algebraic analysis of the finite order
variation of Hodge structure, but it  does not pass through a proof of  Theorem \ref{theodonagireform}. Theorem \ref{theodonagireform} tells that for the given pairs $(d,n)$, a hypersurface of degree $d$ and dimension $n-1$  can be reconstructed  from its first order variation of Hodge structure. Instead, our proof will involve the higher order variation of Hodge structure.

  We introduce in this paper a main new ingredient, which is the notion of {\it Schiffer variation of a hypersurface}  (see Section \ref{secschiffer}). These Schiffer variations are of the form
 \begin{eqnarray}\label{eqfirstschiffer} f_t=f+t x^d
 \end{eqnarray}
 (up to a change of variable $t$)
 and we believe they are interesting for their own. The chosen  terminology
 comes from the notion of Schiffer variations for a smooth curve  $C$. They
 consist  in deforming the complex structure of  $C$  in a way that is
 supported on a point $p$  of $C$. First order Schiffer variations  are in that case
 the elements $u_p\in \mathbb{P}(H^1(C,T_C))$ given by
 $$[H^0(C,2K_C(-p)) ]\in\mathbb{P}(H^0(C,2K_C)^*)=\mathbb{P}(H^1(C,T_C)).$$
 First order   Schiffer variations (\ref{eqfirstschiffer}) of hypersurfaces $X_f$ are the tangent directions at $0$ of Schiffer variations of $f$. They   are parameterized by
the  $d$-th Veronese embedding of $\mathbb{P}(S^1)$ in $\mathbb{P}(S^d)$ projected to $\mathbb{P}(R^d_f)$ via the linear projection
$\mathbb{P}(S^d)\dashrightarrow\mathbb{P}(R^d_f)$.  Although the following result is easy to prove, it is crucial for our strategy.
\begin{prop}\label{propintroschi}  (Cf. Proposition  \ref{proschifferrec}.) Let $X_f,\,X_g$ be two smooth  hypersurfaces of degree $d\geq4$ and dimension $n-1\geq 3$, with $f$ generic. If there exists a linear  isomorphism $R_f^d\cong R_g^d$ mapping the set of first order Schiffer variations of $f$ to  the set of first order Schiffer variations of $g$, $X_f$ is isomorphic to $X_g$.
\end{prop}

Our strategy then  consists in  characterizing  Schiffer variations by the formal properties of the variation of Hodge structure along them. An obvious but key point (see Lemma \ref{leconstant}) is  the fact that the structure of the Jacobian ring (hence of the infinitesimal variation of Hodge structure)
does not change much along them. This follows from the fact that the Jacobian ideals of $f$ and $f_t=f+tx^d$ agree modulo the ideal generated by $x^{d-1}$. This is a higher order property since it concerns  the variation of the infinitesimal variation of Hodge structures.  It would be nice to have a better   understanding and a more  Hodge-theoretic, less formal, characterization  of Schiffer variations.
\begin{rema}{\rm Schiffer variations have very special first order properties mentioned above, as their tangent vector at any point lies in the Veronese variety. However, crucial to our argument is the fact that they also satisfy  higher order conditions,  saying that the hypersurface in $\mathbb{P}^{n-1}$ defined by $f_{\mid x=0} $ is constant, independent of $t$. }
\end{rema}
Our main result can be rephrased as follows.
\begin{theo} (Cf. Claim \ref{propvero}.) \label{theointronew} Let $d$ be large enough, and $n\geq d$. Let $f\in U_{d,n}$ be generic, $ U\subset U_{d,n}$,  $V\subset U_{d,n}$ be Euclidean open sets with $f\in U$, and let $i:U\cong V$ be a holomorphic diffeomorphism inducing an isomorphism of complex variations of Hodge structures
$$ (H^{n-1}_{\mathbb{C},{\rm prim}}, F^p\mathcal{H}^{n-1})\cong i^{-1}(H^{n-1}_{\mathbb{C},{\rm prim}}, F^p\mathcal{H}^{n-1})$$
on $U$. Then for a Schiffer variation $(f_t)_{t\in \Delta}$ of $f$ contained in $U$, where $\Delta$ is a disc, with tangent vector $\phi=\frac{df_t}{dt}_{|t=0}\in T_{U,f}$, $i_*\phi$ is a first order Schiffer variation of $f':=i(f)$.
\end{theo}
This theorem easily implies   Theorem \ref{theodonagivoisin} (1) using Proposition \ref{propintroschi}. With more work, it could be improved in two ways:

1)  Under the same assumptions as above, for  a general  Schiffer variation $(f_t)_{t\in \Delta}$ of $f$, $(i(f_t))_{t\in \Delta}$ is a Schiffer variation of $f'=i(f)$.

2) One should be able to replace the Schiffer variation $(f_t)_{t\in \Delta}$ of $f$ parameterized by a disc by a second order Schiffer variation of $f$.

\vspace{0.5cm}

The paper is organized as follows. In Section \ref{secpolystruc},  we first   explain   (see  Section \ref{secnounou}) how  Theorem \ref{theodonagireform} implies Theorem \ref{theodonagi} and we next discuss the notion of polynomial structure on the data of an infinitesimal variation of Hodge structure of a hypersurface. We discuss various  recipes toward proving uniqueness of the polynomial structure, including Donagi's method. For  example, we exhibit a very simple  recipe to show that the natural   polynomial structure for  most hypersurfaces of degree
$d$ dividing $n+1$
is rigid.  In Section \ref{sec46}, we prove  case (2)  of Theorem \ref{theodonagivoisin}, that is,  degree 4. This
proof follows Donagi's argument but provides a different recipe to prove the uniqueness of the polynomial structure of the infinitesimal variation of Hodge structure in these cases.

The main new ideas and results of the paper appear starting from   Section \ref{secschiffer} where we introduce Schiffer variations of hypersurfaces and discuss their formal properties. The proof of Case (1) of  Theorem \ref{theodonagivoisin} is given in Section \ref{secproofmain},
where we give a characterization of Schiffer variations based on the local analysis of the  variation of Hodge structure along  the corresponding family of  hypersurfaces and prove Theorem \ref{theointronew} (see  Proposition \ref{proschigen} and Claim \ref{propvero}).

\vspace{0.5cm}

{\bf Thanks.} {\it I thank Nick Shepherd-Barron for reminding me that  the Donagi method a priori works only   starting from $n=3$ and that  Theorem \ref{theodonagi} is actually wrong for $(d,n)=(4,2)$. I also thank the referees for their very careful reading and constructive criticism.
 This work was started at MSRI during the program ``Birational Geometry and Moduli Spaces'' in the Spring 2019.
I thank the organizers for inviting me to stay there  and the Clay Institute for its generous support.}

\section{Polynomial structure and the Torelli theorem \label{secpolystruc}}

\subsection{Donagi's strategy and reduction to Theorem \ref{theodonagireform}  \label{secnounou}}
For completeness, and because this argument will be also used in the last section, we explain in this section how Theorem \ref{theodonagireform} implies
Theorem \ref{theodonagi}.
Assume $f\in U_{d,n}$ is very general and $X_{f'}$ is a smooth  hypersurface of degree $d$ and dimension $n-1$ with   polarized Hodge  structure on $H^{n-1}(X_{f'},\mathbb{Q})_{\rm prim}$,  isomorphic to the Hodge structure on  $H^{n-1}(X_f,\mathbb{Q})_{\rm prim}$. We claim that this implies, except in the case where $(d,n)=(3,3)$, that  $X_{f'}$ is also very general and there
is an isomorphism of variations of Hodge structures on respective  neighborhoods $U$, $V$ of
$f$  and $f'$ in their  moduli space $U_{d,n}$.
Indeed, the Hodge locus $$\Gamma_\phi\subset U\times V\subset U_{d,n}\times U_{d,n}$$
  defined as the set of points $t,\,t'\in U\times V$ such  that  the given  isomorphism
 $$\phi:H^{n-1}(X_f,\mathbb{Q})_{\rm prim}\cong H^{n-1}(X_{f'},\mathbb{Q})_{\rm prim}$$ induces an isomorphism of Hodge structures
 $$H^{n-1}(X_t,\mathbb{Q})_{\rm prim}\cong H^{n-1}(X'_{t'},\mathbb{Q})_{\rm prim}$$  is by the Cattani-Deligne-Kaplan theorem
 \cite{CDK} the restriction to $U\times V$ of a closed algebraic subset (that we also denote $\Gamma_\phi$)  of $U_{d,n}\times U_{d,n}$. As $f$ is very general and in the image of ${\rm pr}_1: \Gamma_\phi\rightarrow U_{d,n}$, ${\rm pr}_1$ has to be dominant.

 An important point is the fact that, as an easy consequence of Macaulay theorem \cite[Theorem 6.19]{voisinbook},  the smooth  hypersurfaces of degree $d$ in $\mathbb{P}^n$, with $(d,n)\not=(3,3)$   satisfy the infinitesimal Torelli theorem. This means that the period map is an immersion
at the points $f$ of the open set $U_{d,n}^0\subset U_{d,n}$  parameterizing automorphisms free hypersurfaces,  once $(d,n)\not=(3,3)$. As $f$ is very general, we can assume that $f$ belongs to $U_{d,n}^0$.  It follows that  the projection
$${\rm pr}_2: \Gamma_\phi \rightarrow V$$ must be locally finite since, by definition, the fiber of ${\rm pr}_2$ over any $t'\in V$ parameterizes points $t$ with isomorphic Hodge structures on $H^{n-1}(X_t,\mathbb{Q})_{\rm prim}$.  In particular, $f'$ is also very general so we can assume that $f'$ is also automorphism free.

The two projections  ${\rm pr}_1$, ${\rm pr_2}$ are thus  immersions and  dominant morphisms, hence  they must be  \'etale.   It thus follows that
$\Gamma_\phi$  induces a local holomorphic diffeomorphism $i$ between $U$ and $V$ which, by definition of $\Gamma_\phi$, has the property that the isomorphism
$$\phi:H^{n-1}_{\mathbb{Q},{\rm prim}}\rightarrow i^{-1}H^{n-1}_{\mathbb{Q},{\rm prim}}$$
of trivial local systems on $U$
induces an isomorphism of variations of Hodge structures. Here, if $\pi:\mathcal{X}_{d,n}\rightarrow U_{d,n}^0$ is the universal hypersurface,
 $H^{n-1}_\mathbb{C}$ is the local system $R^{n-1}\pi_*\mathbb{C}_{\rm prim}$ on $U_{d,n}^0$.
 Taking the differential of this isomorphism provides a commutative diagram
where the vertical maps are isomorphisms
\begin{eqnarray}\label{numerodiagavechpq}
 \xymatrix{
&R^d_f\ar[r]\ar[d]&\bigoplus_{p+q=n-1} {\rm Hom}\,(H^{p,q}(X_f)_{\rm prim},H^{p-1,q+1}(X_f)_{\rm prim})\ar[d]\\
&R^d_{f'}\ar[r]& \bigoplus_{p+q=n-1}{\rm Hom}\,(H^{p,q}(X_{f'})_{\rm prim},H^{p-1,q+1}(X_{f'})_{\rm prim})
,}
\end{eqnarray}
where the vertical map on the left is the differential  $i_*$ at $f\in U$ and the vertical map on the right is induced by the isomorphism of Hodge structures
$\phi: H^{n-1}(X_f,\mathbb{Q})_{\rm prim}\rightarrow H^{n-1}(X_{f'},\mathbb{Q})_{\rm prim}$.
 By Theorem \ref{theCagri}, we then get a commutative diagram (\ref{numerodiagavecR}) to which Theorem \ref{theodonagireform}  applies.
 \subsection{Polynomial structure and the symmetrizer lemma}
The method  used by Donagi to prove Theorem \ref{theodonagireform} consists in applying the ``symmetrizer lemma''
(Proposition  \ref{propsymm} below),
in order to recover from the data  (\ref{eqmulti}) the whole
Jacobian ring in degrees divisible by  $l$, where $l$ is the g.c.d. of $n+1$ and $d$. This result  proved first
 in \cite{donagi} for the Jacobian ring of generic hypersurfaces, and  reproved in \cite{greendonagi} for any smooth hypersurface (and more generally quotients $R_{f_{\bullet}}$ of the polynomial ring $S=\mathbb{C}[X_0,\ldots,X_n]$ by a regular sequence $f_{\bullet}=(f_0,\ldots,f_n)$ with ${\rm deg}\,f_i=d-1$),
is  the following
statement.
Consider the multiplication map
\begin{eqnarray}\label{eqmultkkprime} R^k_{f_{\bullet}}\otimes R^{k'}_{f_{\bullet}}\rightarrow R^{k+k'}_{f_{\bullet}},\\
\nonumber
a\otimes b\mapsto ab.
\end{eqnarray}

\begin{prop} \label{propsymm} Let $N=(n+1)(d-2)$. Then, if   ${\rm Max}\,(k,N-k')\geq d-1$ and $N-k-k'>0$,
 the multiplication map
$$R^{k'-k}_{f_{\bullet}}\otimes R^{k}_{f_{\bullet}}\rightarrow R^{k'}_{f_{\bullet}}$$

is determined by
the multiplication map (\ref{eqmultkkprime})
as follows
\begin{eqnarray}\label{eqsymmetrizer}
R^{k'-k}_{f_{\bullet}}=\{h\in{\rm Hom}\,(R^k_{f_{\bullet}},R^{k'}_{f_{\bullet}}),\,bh(a)=ah(b)\,\,{\rm in}\,\,R^{k+k'}_{f_{\bullet}},\,\forall a,\,b\in R^k_{f_{\bullet}}\}.
\end{eqnarray}
\end{prop}
Coming back to the case of a Jacobian ring $R_f$, when $d$ divides $n+1$, the infinitesimal variation of Hodge structure (\ref{eqivhs}) of $X_f$, translated in the form (\ref{eqmulti}), involves only
pieces $R^k_f$ of the Jacobian ring of degree $k$ divisible by $d$. Hence the symmetrizer lemma at best
allows  us, starting from the  IVHS of the hypersurface, to reconstruct the Jacobian ring in degrees divisible by $d$.
 At the opposite, when $d$ and $n+1$ are coprime, repeated applications of the symmetrizer lemma
 allow us to reconstruct the whole Jacobian ring. In degree $<d-1$, the Jacobian ring coincides with the polynomial
 ring, hence we directly recover in that case
 the multiplication map
 $${\rm Sym}^d (S^1)\rightarrow R^d_f$$
 and its kernel $J_f^d$. The  proof of Donagi is then finished by applying Mather-Yau's theorem \cite{mayau} (see also Proposition  \ref{theoyau}).

This leads us to the following definition. Suppose that we have two integers $d,\,n$ and  the partial data of a graded  ring structure $R^*$, namely
finite dimensional vector spaces
$R^d,\,R^{-(n+1)+id}$ for $i$ such that $-(n+1)+id\geq 0$ with multiplication maps
\begin{eqnarray}\label{eqpultipli} \mu_i:  R^d\otimes R^{-(n+1)+id}\rightarrow R^{-(n+1)+(i+1)d}.
\end{eqnarray}
 When  $d$  divides $n+1$,  we  get all the upper-indices divisible by $d$, and an actual  ring structure $R^{d*}$, but in general  (\ref{eqpultipli}) is the sort of data provided by the  infinitesimal variation of Hodge structure of a hypersurface of degree $d$ in $\mathbb{P}^{n}$.
 Let   $S^k$ be the degree $k$ part of the  polynomial ring in $n+1$ variables.
 \begin{Defi}\label{defipoly} A polynomial structure in $n+1$ variables for the partial data of a graded  ring structure
 $$(R^d,\,R^{-(n+1)+id},\mu_i)$$ is the data   of a rank $n+1$ base-point free linear subspace  $J\subset S^{d-1}$  generating a graded  ideal $J^*\subset S^*$, of a linear  isomorphism $S^d/J^d\cong  R^d$ and, for all $i$,   of
  linear  isomorphisms
 $$S^{-(n+1)+id}/J^{-(n+1)+id}\cong R^{-(n+1)+id},$$
compatible with the multiplication maps, i.e.
 making the following diagrams commutative:
 \begin{eqnarray}\label{numerodiag1}
 \label{diagram} \xymatrix{
&S^d\otimes S^{-(n+1)+id}\ar[r]\ar[d]^{\cong}&S^{-(n+1)+(i+1)d}\ar[d]^{\cong}\\
&R^d\otimes R^{-(n+1)+id}\ar[r]^{\mu_i}&R^{-(n+1)+(i+1)d}.}
\end{eqnarray}
 \end{Defi}
The group ${\rm GL}(n+1)$  acts in the obvious way on the set of polynomial structures. We will say that the polynomial structure
of $(R^d,\,R^{-(n+1)+id},\,\mu_i)$ is unique if all its polynomial structures are conjugate under
${\rm GL}(n+1)$. As explained above,  Donagi's Theorem \ref{theodonagireform} has the more precise form that, under some assumptions on $(d,n)$,
the polynomial  structure of the infinitesimal variation of Hodge structure
$(R^d_f,\,R^{-(n+1)+id}_f,\,\mu_i)$ of a smooth hypersurface $X_f$ is unique, and this is sufficient to imply the generic Torelli theorem for hypersurfaces of these degree and dimension.
We will prove a similar statement in the case (2)  (that is degree $d=4$) of Theorem \ref{theodonagivoisin}, at least for generic $f$ and $n$ large enough.

For  the main series of cases not covered by Donagi's theorem, namely when $d$ divides $n+1$, we have not been able to prove the uniqueness of the polynomial structure of $R^{d*}_f$ (even for generic $f$), although it is likely to be true (and it is proved in \cite{voisin} for $d=5$, $n=4$). We conclude this section by the proof of  a weaker statement that provides evidence for the uniqueness.
We will say that a polynomial structure is rigid if its small deformations are given by its  orbit under ${\rm GL}(n+1)$.
We have the following
\begin{prop}\label{prorigid} Assume $n+1\geq 8$ and $d\geq 6$. Let $f\in S^{d}$ be a generic  homogeneous  polynomial of degree
$d$ in $n+1$ variables and $R_{f}^{d*}$ be its Jacobian  ring in degrees divisible by $d$. Then
the natural polynomial structure
$$S^{d*}\rightarrow R^{d*}_{f}$$
given by the quotient map
is rigid.
\end{prop}
\begin{rema}{\rm The case where  $d=4$ and $n+1\cong 2$ mod. $4$  will be studied in next section. We will prove there, using a different recipe, that the polynomial structure on $R^{2*}_f$  is unique for $n$ large enough.
}
\end{rema}
\begin{rema}{\rm Proposition \ref{prorigid} implies that the natural polynomial structure of $R_{f_{\bullet}}^{d*}$ for a generic rank $n+1$ regular sequence $f_{\bullet}$ of degree $d-1$ homogeneous polynomials is rigid.}
\end{rema}
We will use in fact only the multiplication map in degree $d$
$$\mu: R^d_{f}\times R^d_{f}\rightarrow R^{2d}_{f}.$$
  Proposition \ref{prorigid} will  be implied by Proposition
\ref{proleIx} below.
For our original   polynomial structure on $R_f^{d*}$, and for each $x\in S^1$, we get a pair of vector subspaces
\begin{eqnarray}\label{eqideals} I_x^{d}:=xR^{d-1}_{f}\subset R^{d}_{f}, \,\,I_x^{2d}:=xR^{2d-1}_{f}\subset R^{2d}_{f},
\end{eqnarray}
which form an ideal in the sense that
\begin{eqnarray} \label{eqconideal} R^d_{f}I_x^d\subset I_x^{2d}.
\end{eqnarray}
It is not hard  to see  that the multiplication map by $x$, from $R^{d-1}_{f}$ to $R^{d}_{f}$, is  injective for a generic  $x\in S^1$  when
$f$ is  generic with   $d\geq 4$ and $n\geq 3$ (or $d\geq 3$ and $n\geq 5$). In fact, we even have (statement (ii) will be used only later on)
\begin{lemm}\label{claimpourplustardaussi} (i) The multiplication map by $x$ is injective on $R^{2d-1}_{f}$ when $f$ is generic, $x\in S^1$ is generic and  \begin{eqnarray}
\label{eqpourineqnewfinal} 2(2d-1)<(d-2)(n+1)
\end{eqnarray} (for example, $n+1\geq 5$ and $d>8$, or $n+1\geq 6$ and $d>4$ work).

(ii) The multiplication map by $x$  is injective  on $R^{3d-1}_{f}$ when $f$ is generic, $x\in S^1$ is generic and  $$2(3d-1)<(d-2)(n+1)$$  (for example, $n+1\geq 5$ and $d>8$, or $n+1\geq 6$ and $d>4$ work).

(iii) The multiplication map by $x^l$  is injective  on $R^{k}_{f}$ when $f$ is generic, $x\in S^1$ is generic and  $$2k+l\leq (d-2)(n+1).$$
\end{lemm}
\begin{proof} As the dimensions of the vector spaces $R^k_f$ are independent of $f$ (assumed  to define a smooth hypersurface), the conclusions are open properties of $f$, hence it suffices to check them for a particular $f$. Take for $f$  the  Fermat polynomial  $f_{\rm Fermat}=\sum_{i=0}^nX_i^d$. Then
$R_{f_{\rm Fermat}}^*$ identifies with the cohomology  ring $H^{2*}((\mathbb{P}^{d-2})^{n+1},\mathbb{C})$ (indeed, it has generators $X_i$ and relations $X_i^{d-1}=0$) and $x=\sum_iX_i$
corresponds to an ample class in $H^{2}((\mathbb{P}^{d-2})^{n+1},\mathbb{C})$. By the  hard Leschetz theorem for $(\mathbb{P}^{d-2})^{n+1}$,
the multiplication by $x$ is thus injective on $R^{2d-1}_{f_{\rm Fermat}}$ if $2(2d-1)<(d-2)(n+1)$,
and injective on $R^{3d-1}_{f_{\rm Fermat}}$ if $2(3d-1)<(d-2)(n+1)$. More generally, the Lefschetz isomorphism for the
power $x^l$ gives the injectivity of $x^l$  on $R^k_{f_{\rm Fermat}}$ when $2k+l\leq (d-2)(n+1)$.
\end{proof}
\begin{rema}{\rm
The estimate in (i) is optimal for dimension reasons. Indeed, the dimensions of the graded pieces $R^{k}_{f}$ are increasing  in the interval $k\leq \frac{(d-2)(n+1)}{2}$, and decreasing in the interval  $ \frac{(d-2)(n+1)}{2}\leq k\leq (d-2)(n+1)$.}
\end{rema}
It  follows from Lemma \ref{claimpourplustardaussi} that, assuming inequality  (\ref{eqpourineqnewfinal}), the space  $I_x^d$ defined in (\ref{eqideals}) has generic  dimension
$r_{d-1}:={\rm dim}\, R^{d-1}_{f} $, while $I_x^{2d}$ has generic dimension $r_{2d-1}:={\rm dim}\, R^{2d-1}_{f} $.

\begin{prop}\label{proleIx} If $f$ is generic of degree $d$ in $n+1$ variables, and  $d\geq 6,\,\,n\geq 9$, then  the
subset    $Z_{\rm ideal}=\{[I^d_x]\in G(r_{d-1},R^d_{f}),\,x\in S^1\}$ of the Grassmannian $G(r_{d-1},R^d_{f})$  is a reduced  component of
the  closed algebraic subset  $Z\subset
G(r_{d-1},R^d_{f})$ defined as
\begin{eqnarray}\label{eqdefiZideal} Z=\{[W]\in G(r_{d-1},R^d_{f}),\,{\rm dim}\,(R^d_{f}\cdot W)\leq r_{2d-1}\}.
\end{eqnarray}
\end{prop}
\begin{proof}
 The tangent space to $Z_{\rm ideal}$ at the point   $[I^d_x]\in G(r_{d-1},R^d_{f})$ is
the image of $S^1/\langle x\rangle$ in ${\rm Hom}\,(R^{d-1}_{f},R^d_{f}/xR^{d-1}_{f})=T_{G(r_{d-1},R^d_{f}),[I_x^d]}$ given by multiplication by $y\in S^1/\langle x\rangle$, where we identify $I_x^d$ with $R^{d-1}_{f}$ via multiplication by $x$.
Let us now compute the Zariski tangent space  to $Z$ at $[I^d_x]$ for $f$ and $x$ generic. As ${\rm dim}\,I_x^{2d}=r_{2d-1}$ is maximal by the claim above, the condition (\ref{eqdefiZideal}) provides the following infinitesimal conditions:
\begin{eqnarray}\label{eqdefitangentZideal}
T_{Z,[I_x^d]}=\{h\in {\rm Hom}\,(R^{d-1}_{f},R^d_{f}/xR^{d-1}_{f}),\, \sum_iA_ih(B_i)=0\,\,{\rm in}\,\,R^{2d}_{f}/xR^{d-1}_{f},
\\
\nonumber
{\rm for \,\,any}\,\,\,K=\sum_iA_i\otimes B_i\in R^d_f\otimes R^{d-1}_{f}\,\,{\rm such\,\,that}\,\,\sum_iA_iB_i=0\,\,{\rm in} \,\,R^{2d-1}_{f}\}.
\end{eqnarray}
Equation (\ref{eqdefitangentZideal}) says that $h: R^{d-1}_{f}\rightarrow R^d_{f}/xR^{d-1}_{f}$ is a ``morphism of
$R^d_{f}$-modules'', the set of which we will denote by
${\rm Mor}_{R^d_f}(R^{d-1}_{f},R^d_{f}/xR^{d-1}_{f})$, in the sense that we have a commutative diagram for some $h'\in {\rm Hom}\,(R^{2d-1}_{f},R^{2d}_{f}/\langle x\rangle)$
\begin{eqnarray}\label{numerodiagpouridealx}
 \xymatrix{
&R^d_{f}\otimes R^{d-1}_{f}\ar[r]\ar[d]^{Id\otimes h}& R^{2d-1}_{f} \ar[d]^{h'}\\
&R^d_{{f}}\otimes R^{d}_{f}/\langle x\rangle\ar[r]& R^{2d}_{f}/\langle x\rangle
,}
\end{eqnarray}
where the horizontal maps are given by multiplication.
The equality $T_{Z_{\rm ideal}}=T_Z$ at the point $[I_x^d]$ is thus equivalent to the fact  that all the ``$R^d_{f}$-modules morphisms''
$h: R^{d-1}_{f}\rightarrow R^d_{f}/xR^{d-1}_{f}$, are given by multiplication by some $y\in S^1$, followed by reduction mod $x$.
This is the statement of the following
\begin{lemm}\label{lepourmorRmod} Let ${f}$ be a generic  homogeneous  degree $d$ polynomial  in $n+1$ variables with $d \geq 5,\,n\geq 9$   (or $d\geq 6$ and $n\geq 7$), and let $x\in S^1$ be generic.
Then the natural map $S^1/\langle x\rangle\rightarrow {\rm Mor}_{R_{f}^d}(R^{d-1}_{f},R^d_{f}/\langle x\rangle)$ is  surjective.
\end{lemm}
\begin{proof} The existence of  $h'$ as in  (\ref{numerodiagpouridealx}) says that
for any tensor $\sum_i A_i\otimes B_i\in R^d_{f}\otimes R^{d-1}_{f}$ such that
$\sum_iA_iB_i=0$ in $R^{2d-1}_{f}$,
$\sum_i A_ih(B_i)=0$ in $R^{2d}_{f}/\langle x\rangle$.
\begin{claim} \label{claimnouveau} Under the same assumptions as in Lemma \ref{lepourmorRmod}, for a generic $q\in R^{d-1}_{f}$, the multiplication map
$$q:R^{d+1}_{f}/\langle x\rangle\rightarrow R^{2d}_{f}/\langle x\rangle$$
is injective.
\end{claim}
\begin{proof}  This is proved again by looking at
the Fermat polynomial $f_{\rm Fermat}=\sum_iX_i^d$ and choosing carefully $x$ so that multiplication by $x$ is injective on
 $R^{2d-1}_{f_{\rm Fermat}}$, and multiplication by $q$ is injective on
 $R^{d+1}_{f_{\rm Fermat}}/\langle x\rangle$. We write
$f_{\rm Fermat}=f'_{\rm Fermat}+f''_{\rm Fermat}$, where $f'_{\rm Fermat}=\sum_{i=0}^{4}X_i^d$ and $f''_{\rm Fermat}=\sum_{i=5}^{n}X_i^d$. We take $x=\sum_{i=0}^4X_i$
and $q=(\sum_{i=5}^NX_i)^{d-1}$. We observe that
$$R_{f_{\rm Fermat}}^*\cong R_{f'_{\rm Fermat}}^*\otimes R_{f''_{\rm Fermat}}^*,$$
as graded rings,
and that $x$ acts by multiplication on the left term $ R_{f'_{\rm Fermat}}^*$, while $q$ acts by multiplication
on the right term $R_{f''_{\rm Fermat}}^*$. So it suffices to show that multiplication by $x$ is injective on
$ R_{f'_{\rm Fermat}}^k$ for $k\leq 2d-1$ and multiplication by $q$ is injective on
$ R_{f''_{\rm Fermat}}^k$ for $k\leq d+1$. The first statement follows from Lemma
\ref{claimpourplustardaussi} (i)  when $2(d+1)<5(d-2)$, hence when $d\geq 5$.
The second  statement holds by  Lemma
\ref{claimpourplustardaussi} (iii) when  $2(d+1)+d-1\leq (n-4)(d-2)$, and in particular if $d\geq 6$ and
$n\geq 9$.
\end{proof}
We deduce from Claim \ref{claimnouveau} that
for any tensor  $\sum_i A_i\otimes B_i\in S^1\otimes R^{d-1}_{f}$ such that
$\sum_iA_iB_i=0$ in $R^{d}_{f}$, we have
\begin{eqnarray}
\label{eqnoyauvan}\sum_i A_ih(B_i)=0\,\,{\rm in}\,\,R^{d+1}_{f}/\langle x\rangle,
\end{eqnarray}
since this becomes true after multiplication by $q$.
It follows now that
$h$ vanishes on $\langle x\rangle$. Indeed, let $b=xb'$. Then for any $y\in S^1$,
we have $yb=xb''$ with $b''=yb'$. Hence  by (\ref{eqnoyauvan}), we get
$yh(b)=xh(b'')=0$. Hence $yh(b)=0$ in $R^{d+1}_f/\langle x\rangle$ for any $y\in S^1$, and it follows,  by choosing $y$ such that multiplication by $y$ is injective on $R^d_{f}/\langle x\rangle$, that
$h(b)=0$ in $R^d_{f}/\langle x\rangle$. Thus $h$ induces a morphism
$$\overline{h}: R^{d-1}_{f}/\langle x\rangle\rightarrow  R^{d}_{f}/\langle x\rangle,$$
which also satisfies (\ref{eqnoyauvan}).
Assuming $d\geq 6,\,n\geq 9$, we now show by similar arguments as above  that for generic  $z,\,y\in S^1/\langle x\rangle$, the following holds.  For any $p,\,q\in R^d_{f}/\langle x\rangle$,
\begin{eqnarray}
\label{eqimplique} y p+zq=0\,\,{\rm in}\,\,R^{d+1}_{f}/\langle x\rangle \Rightarrow  p=zr,\,q=-yr,
\end{eqnarray}
for some $r\in R^{d-1}_f/\langle x\rangle$.
 Furthermore we already  know that  the multiplication map by $z$ from  $ R^d_{f}/\langle x\rangle$ to  $R^{d+1}_{f}/\langle x\rangle$ is injective.
It follows that there exists $$\overline{h}'': R^{d-2}_{f}/ \langle x\rangle\rightarrow  R^{d-1}_{f}/ \langle x\rangle$$
inducing $\overline{h}$, that is,
\begin{eqnarray}
\label{eqdefiprophseconde} \overline{h}(ap)=a\overline{h}''(p)
\end{eqnarray}
for any $p\in R^{d-2}_{f}/ \langle x\rangle$, and any $a\in S^1/\langle x\rangle$.
Indeed, $y$ and $z$ being as above, we have for any $p\in R^{d-2}_{f}/ \langle x\rangle$
$$ y(zp)-z(yp)=0\,\,{\rm in}\,\, R^{d}_{f},$$ hence
by (\ref{eqnoyauvan}), we get that
$y\overline{h}(zp)-z \overline{h}(yp)=0$ in $R^{d+1}_{f}/\langle x\rangle$, and  by (\ref{eqimplique}), this gives
$\overline{h}(zp)=z\overline{h}''(p)$, which defines $\overline{h}''$. One then shows that the map
$\overline{h}''$ so defined does not depend on $z$ and satisfies (\ref{eqdefiprophseconde}), which is easy.
To finish the proof, we construct similarly $\overline{h}''':R^{d-3}_{f}/\langle x\rangle
\rightarrow R^{d-2}_{f}/\langle x\rangle$ inducing $\overline{h}''$  and  $\overline{h}^{iv}:R^{d-4}_{f}/\langle x\rangle
\rightarrow R^{d-3}_{f}/\langle x\rangle$ inducing $\overline{h}'''$. As $R^i_{f}/\langle x\rangle=S^i/\langle x\rangle$ for $i\leq d-2$, it is immediate to show that $\overline{h}^{iv}$ is multiplication by some element of $S^1$, hence also $\overline{h}$.
\end{proof}
The proof of Proposition \ref{proleIx} is thus complete.
\end{proof}
\begin{proof}[Proof of Proposition \ref{prorigid}] Let $f$ be generic of degree $d\geq 6$ in $n+1\geq 10$ variables.
We first claim that for {\it any} $x\in S^1$, the multiplication map by $x:R^{d-1}_{f}\rightarrow R^{d}_{f}$ is injective, and that the morphism
 \begin{eqnarray}
 \label{eqinPhi} \Phi:\mathbb{P}(S^1)\rightarrow G(r_{d-1},R^d_{f}),\,\,x\mapsto xR^{d-1}_{f}\subset R^d_f
 \end{eqnarray}
so constructed is an embedding. None of these statements is  difficult to prove. The first statement says that
 if ${f}$ is a generic homogeneous  degree $d$ polynomials in $n+1$ variables, ${f}$ does not satisfy an
 equation $\partial_u(f)_{\mid H}=0$ for some hyperplane  $H \subset \mathbb{P}^n$ and vector field $u$ on $\mathbb{P}^n$. The obvious  dimension count shows that this holds if $h^0(\mathbb{P}^{n-1},\mathcal{O}(d))>n-1+\frac{(n+1)^2}{2}$, which holds if $d\geq 4,\,n\geq3$.
 As for the second statement, suppose that $xR^{d-1}_f=yR^{d-1}_f$ for some non-proportional
 $x,\,y\in S^1$. Then there is a subspace of dimension
 $\geq {\rm dim}\,S^{d-1}$ of pairs $(p,q)\in S^{d-1}\times S^{d-1}$ such that
 $xp=yq$ in $R^d_f$, that is $xp-yq\in J^d_f$.
 As the kernel of the map $x-y: S^{d-1}\times S^{d-1}\rightarrow S^{d}$ is of dimension ${\rm dim}\,S^{d-2}$,  this would imply that
 \begin{eqnarray}
 \label{eqinegdim}{\rm dim}\,J_f^d\cap {\rm Im}\,(x+y)\geq {\rm dim}\,S^{d-1}-{\rm dim}\,S^{d-2}=h^0(\mathbb{P}^{n-1},\mathcal{O}(d-1)).
 \end{eqnarray}
 As ${\rm dim}\,J_f^d=(n+1)^2$, (\ref{eqinegdim}) is impossible if $h^0(\mathbb{P}^{n-1},\mathcal{O}(d-1))>(n+1)^2$, which holds if $n\geq 5,\,d\geq 4$. We thus proved that the map  $\Phi$ of (\ref{eqinPhi}) is injective. That it is an immersion follows in the same way because the differential at $x$ is given by the multiplication map
 $$y\mapsto \mu_y: R^{d-1}_f\rightarrow R^d_f/xR^{d-1}_f,$$
 and  $\mu_y$ is zero if and only if $y R^{d-1}_f\subset xR^{d-1}_f$, which has just been excluded. The claim is thus proved.

It follows from  the claim  and from  Proposition \ref{proleIx}  that, if we have a family
of polynomial structures
$$\phi_t: S^{d*}\rightarrow R^{d*}_f,$$
with $\phi_0$ the natural one, then there is an isomorphism
$$\psi_t: \mathbb{P}(S^1)\cong  \mathbb{P}(S^{1}),$$
such that for any
$x\in S^{1}$,
$$\phi_t(x S^{d-1})=\psi_t(x)R^{d-1}_f.$$  Such a  projective isomorphism
is induced by a linear isomorphism
$$\tilde{\psi}_t: S^{1}\cong  S^{1},$$
and composing $\phi_t$ with the automorphism of $S^{d*}$  induced by  $\tilde{\psi}_t^{-1}$, we conclude that
we may assume that for any
$x\in S^{1}$,
\begin{eqnarray}\label{equationtransformee} \phi_t(x S^{d-1})=xR^{d-1}_f.
\end{eqnarray}
We claim that this implies $\phi_t: S^d\rightarrow R^d_f$ is the natural map of reduction  mod $J_f$.
To see this, choose a general $x$, so that the multiplication map by $x$ is injective on $R^{2d-1}_f$. The
polynomial structure given by $\phi_t$ and satisfying (\ref{equationtransformee})
provides two  linear maps
$$\phi'_t: S^{d-1}\rightarrow R^{d-1}_f,\,\phi''_t: S^{2d-1}\rightarrow R^{2d-1}_f,$$
such that $x\phi'_t=\phi_t\circ x: S^{d-1}\rightarrow R^d_f$, $x\phi''_t=\phi_t\circ x:S^{2d-1}\rightarrow R^{2d}_f$,
and
the injectivity of the map of multiplication by $x$ on $R^{2d-1}_f$ implies that the following diagram commutes, since it commutes after multiplying the maps by $x$.
\begin{eqnarray}\label{numerodiag27011440}
 \xymatrix{
&S^{d-1}\otimes S^d\ar[r]\ar[d]^{\phi'_t\otimes \phi_t}& S^{2d-1}\ar[d]^{\phi''_t}\\
&R^{d-1}_f\otimes R^d_f\ar[r]& R^{2d-1}_f
.}
\end{eqnarray}
The horizontal maps in the diagram above are the multiplication maps.
Following Donagi \cite{donagi}, the multiplication map on the bottom line determines the
polynomial structure of $R_f^*$, because it determines (for $d\geq3$)
$S^1$ and the multiplication map $S^1\otimes R^{d-1}_f\rightarrow R^d_f$ by the symmetrizer lemma (Proposition \ref{propsymm}).
The diagram (\ref{numerodiag27011440}) then says that
up to the action of an automorphism $g$ of $S^*$, the polynomial structure given by $(\phi'_t,\,\phi_t)$ is the standard one. Finally, as $g$ must act  trivially on the space $Z_{\rm ideal}$ of ideals by (\ref{equationtransformee}), $g$ is  proportional to the identity.
\end{proof}

\section{The case of degree $4$  \label{sec46}}
We explain in this section how to recover the polynomial structure of the infinitesimal variation of Hodge structure of a generic hypersurface of degree $4$
so as to prove  Theorem   \ref{theodonagivoisin} (2),  namely the cases where
 $d=4$, $n=4m+1$,  with $m$ large. Note that the methods of Schiffer variations that we will develop later would presumably also apply to this case, but it is  much more difficult and does not prove Theorem \ref{theodonagireform} (saying that one can recover a hypersurface from its IVHS).

The congruence conditions is  equivalent   to the fact that we have  ${\rm gcd}(4,n+1)=2$.
The infinitesimal variation of Hodge structure (as translated in (\ref{eqmulti}) using Theorem \ref{theCagri})
\begin{eqnarray}
\label{eqpourvard4} R_f^4\rightarrow \oplus_l {\rm Hom}( R^{4l-n-1}_f, R^{4(l+1)-n-1}_f),
\end{eqnarray}
has for smallest degree term the multiplication map
$$R^{4}_f\otimes R^{2}_f\rightarrow R^{6}_f$$
and  the  symmetrizer lemma  (see  Proposition \ref{propsymm})  allows us to reconstruct
in these cases the whole ring
$R^{2*}_f$, and in particular the multiplication map
\begin{eqnarray}\label{eqmultpetitdeg}  R^{2}_f\otimes R^{2}_f\rightarrow R^4_f.
\end{eqnarray}
(Note that $R^{2}_f=S^{2}$.)
We thus  only have to explain  in both cases how to recover the polynomial structure of (\ref{eqpourvard4}) from (\ref{eqmultpetitdeg}), at least for a generic polynomial $f$.
We use  the notation   $\mathcal{S}q^{2l}_f\subset R^{2l}_f$ for  the set of squares
$$\mathcal{S}q^{2l}_f=\{A^2,\,A\in R^l_f\}\subset  R^{2l}_f.$$
This is a closed algebraic subset which is a cone in $R_f^{2l}$ and we will denote by $\mathbb{P}(\mathcal{S}q^{2l})$ the corresponding closed algebraic subset of $\mathbb{P}(R^{2l}_f)$.
When $d=4,\,d'=2$, (\ref{eqmultpetitdeg}) determines $\mathcal{S}q^{4}_f$. Our proof of Theorem \ref{theodonagivoisin} (2) will be based (following Donagi's strategy described in the previous section) on the following
\begin{claim} \label{clanew21} The algebraic subset
$\mathcal{S}q^{2}_f\subset  R^{2}_f=S^2$ determines the   polynomial structure of the even degree Jacobian ring $R^{2*}_f$ of $f$.
\end{claim}
\begin{proof} Indeed,  passing to the projectivization of these affine cones,  $\mathbb{P}(\mathcal{S}q^{2}_f)$ is the
second Veronese embedding of $\mathbb{P}(S^1)$ in $\mathbb{P}(S^2)$.  Thus the positive generator $H$
of ${\rm Pic}\,(\mathbb{P}(\mathcal{S}q^{2}_f))$ satisfies the property that $H^0(\mathbb{P}(\mathcal{S}q^{2}_f), H)=:V$ has dimension $n+1$ and the  restriction map
$(S^2)^*\rightarrow {\rm Sym}^2V$ is an isomorphism. The dual isomorphism gives the desired isomorphism
${\rm Sym}^2S^1\cong S^2$, with $S^1:=V^*$.
\end{proof}

We observe now  that
the closed algebraic subset $ \mathcal{S}q^{2}_f\subset R^2_f$  has the following property
\begin{eqnarray} \label{eqcarre4} \forall A,\,B\in   \mathcal{S}q^{2}_f,\,AB\in  \mathcal{S}q^{4}_f,
\end{eqnarray}

We prove now the  following result, which by  Claim \ref{clanew21} concludes the proof of Theorem \ref{theodonagivoisin} (2).
\begin{prop}\label{procarre46}
 Let $f$ be a generic homogeneous polynomial of degree $4$ in $n+1$ variables, with $n\geq 599$. Then
the only subvariety $T\subset R^2_f=S^2$ of dimension $\geq n+1$  satisfying the condition
$$AB\in \mathcal{S}q^{4}_f\,\,{\rm for\,\,any}\,\,A,\,B\in T$$
is $ \mathcal{S}q^{2}_f$.
\end{prop}
Note that in this statement, we can clearly assume that $T$ is a cone, since the conditions are homogeneous.
\begin{proof}[Proof of Proposition \ref{procarre46}] We observe that by a proper specialization
 argument, the schematic version of the  statement, saying  that, furthermore, equations (\ref{eqcarre4}) define, at least generically, the reduced structure of   $ \mathcal{S}q^{2}_f$,   is an open condition on the set of polynomials $f$ for which
$R_f^4$, or equivalently $J^4_f$,  has the right dimension.
We  will thus prove this schematic version
for one specific $f$, for which $J^4_f$ has the right dimension.

Let us first explain the specific polynomials we will use.
We will first choose
general linear sections $\mathbb{P}\cap Pf_4$ of the Pfaffian quartic $\mathcal{P}f_4\subset \mathbb{P}(\bigwedge^2 V_8)$,
where $\mathbb{P}\subset \mathbb{P}(\bigwedge^2 V_8)$ is a linear subspace of dimension
$23$ or $24$.
We get this way polynomials $f_i$ of degree $4$  in $24$ or $25$  variables.
In higher dimension, we will then consider polynomials  of the form
$$f=f_{1}(X_{1,1},\ldots,X_{1,i_1})+\ldots+f_{l}(X_{l,1},\ldots,X_{l,i_l}),$$ with $i_1,\ldots,\,i_l\in\{24,\,25\}$,
which  allows us to construct degree $4$  polynomials with
any number $n+1$  of variables starting from $600$.

If a degree $d$ polynomial $f$ defines a smooth hypersurface $X_f$, the Jacobian ideal $J_f$ is generated by the regular sequence $J_f^{d-1}$ of degree $d-1$ polynomials, hence, assuming  $d\geq3$, we find that the multiplication map $S^1\otimes J_f^{d-1}\rightarrow J_f^d$ is an isomorphism, so that
${\rm dim}\,J_f^d=(n+1)^2$.  In general, assuming $X_f$ reduced with regular locus $X_{f,\rm reg}$ and   using the normal bundle sequence
\begin{eqnarray}\label{eqnobuseq} 0\rightarrow \mathcal{T}_{X_{f,{\rm reg}}}\rightarrow T_{\mathbb{P}^n|X_{f,{\rm reg}}}\rightarrow \mathcal{O}_{X_{f,{\rm reg}}}(d)\rightarrow 0,
\end{eqnarray}
we see that the kernel of the map $S^1\otimes J_f^{d-1}\rightarrow J_f^d$ identifies naturally with the set of infinitesimal automorphisms of $X_f$ induced by an infinitesimal automorphism of $\mathbb{P}^n$. The hypersurfaces $X_f$ discussed above are singular. We nevertheless have
\begin{lemm} \label{ledimpourJd} For a polynomial $f$ of the form above, $J_f^4$ has the right  dimension, i.e. $(n+1)^2$.
\end{lemm}
\begin{proof}   One has $f=\sum_j f_{j}$, where each $f_{j}$ involves variables $X_{j,1},\ldots, X^{j,i_j}$. It is immediate to check that the statement for each $f_{j}$ implies the statement for $f$. Turning to the $f_{j}$, they are either general linear sections  of the  quartic Pfaffian hypersurface
in $\mathbb{P}^{27}$
by a  $\mathbb{P}^n$, $n=23$ or $24$.
Let us show that each of them has no infinitesimal automorphism. The automorphism group of the general
  Pfaffian hypersurface
$\mathcal{P}f_k\subset \mathbb{P}(\bigwedge^2V_{2k})$
 is the  group ${\rm {P}GL}(2k)$. We claim that the  automorphism group of a general linear section of dimension
 $>2(2k-2)={\rm dim}\,G(2,V_{2k})$ is also contained in   ${\rm {P}GL}(2k)$. This follows from the fact that after blowing-up   $\mathcal{P}f_{k}$ along its singular
 locus, which parameterizes forms of rank $<2k-2$,  we get a dominant morphism $\widetilde{\mathcal{P}f_{k}}\rightarrow G(2,V_{2k})$, which to a degenerate form associates its kernel.  If we consider a general linear section $X_l$ of $\mathcal{P}f_{2k}$ of dimension $>{\rm dim}\,G(2,V_{2k})$ defined by
 a $r$-dimensional vector  subspace  $W\subset \bigwedge^{2}V_{2k}^*$, the same remains true and
 we get a morphism $\widetilde{X}_l\rightarrow G(2,V_{2k})$ which is dominant with connected  fiber of positive dimension. Thus
 the automorphism group of $X_l$ has to act on $G(2,V_{2k})$ and it has to identify with the group of automorphisms of
 $G(2,V_{2k})$, or automorphisms of $\mathbb{P}(V_{2k})$ preserving the space $W\subset \bigwedge^{2}V_{2k}^*$. It is easy to check that this space is zero once $r\geq3$.
 Coming back to our situation where $k=d=4$, our choices of $r$ are  $r=3$ or $r=4$ for $k=4$. In all cases, the variety $X_l$ has dimension $>2(2k-2)$ so the analysis above applies.
\end{proof}
We now prove Proposition \ref{procarre46}  for $f$ as above.
\begin{lemm}\label{lepour4} Let $f=\sum_jf_j$ be a polynomial of degree $4$ in $n+1$ variables as constructed above. Then
if $T\subset S^2$ is an algebraic subvariety of dimension $\geq n+1$ such that
$$AB\in \mathcal{S}q^4_f\subset R^4_f$$ for any $A,\,B\in  T$, $T= \mathcal{S}q^2$.
\end{lemm}
\begin{proof}
Let as above $f=\sum_l f_{j}$. The singular locus $Z_f$ of $X_f=V(f)$ is the join of the singular loci
$Z_{j}$ of $V(f_j)$ in $\mathbb{P}^{i_j}$. This means that, introducing the natural rational projection map
$\pi:\mathbb{P}^{\sum_j i_j-1}\dashrightarrow\prod_l\mathbb{P}^{i_j-1}$, one has  $Z_f=\pi^{-1}(\prod_jZ_{j})$.
\begin{claim}\label{claim1pour4} The varieties $Z_{f}$ are not contained in any quadric.
\end{claim}
\begin{proof} Consider first the case of the Pfaffian linear sections $Z_j$. The claim follows in this case  because they are general linear sections
of the singular locus $Z$ of the quartic universal Pfaffian $\mathcal{P}f_4$ in $\mathbb{P}(\bigwedge^2 V_8)$, which is defined by the equations
$\omega^3=0$, that is, by cubics, and is not contained in any quadric. The last point can be seen by looking at the  singular locus of $Z$, which consists of forms of rank $2$, that is the Grassmannian $G(2,V_8^*)$. Along this locus, the Zariski tangent space of $Z$
is the full Zariski tangent space of $\mathbb{P}(\bigwedge^2V_8)$. A quadric containing $Z$ should thus be singular
along ${\rm Sing}\,Z$. But ${\rm Sing}\,Z=G(2,V_8^*)$ is not contained in any proper linear subspace of $\mathbb{P}(\bigwedge^2 V_8)$. It follows that
$Z$ is not contained in any quadric. It remains to conclude that the same statement is true for the general linear section $\mathbb{P}^{i_l-1}\cap \mathcal{P}f_4$, with $i_l=24,\,25$. Its  singular locus $Z_l$ is the general linear
section  $\mathbb{P}^{i_l-1}\cap Z$, and we show inductively that any quadric containing $Z_l$ is the restriction of a quadric containing $Z$. This statement only needs that $Z_l$ is non-empty (it has dimension $\geq 18$ in our case)
and that all the successive linear sections $\mathbb{P}^{j}\cap Z$, with $j\geq i_l$, are linearly normal in $\mathbb{P}^{j}$, which is not hard to prove. Finally we have
to show that the same is true for a general $f=\sum f_j$. As already mentioned, $Z_f$ is then a join
$Z_1*\ldots *Z_l$ and a join of varieties not contained in any quadric is not contained in any quadric.
\end{proof}

 We also prove the following
\begin{claim}\label{claim2pour4} (a) The restriction map $S^1\rightarrow H^0(Z_f,\mathcal{O}_{Z_f}(1))$ is an isomorphism.

 (b) The only
$n$-dimensional family $\{D_A\}$ of divisors   on $Z_f$ such that for  some fixed effective divisor $D_0$,  $$2D_A+ D_0\in |\mathcal{O}_{Z_f}(2)|$$  is the family of hyperplane sections of $Z_f$.
\end{claim}
\begin{proof}
 We know  that  $Z_f$ is the join of the $Z_j\subset \mathbb{P}^{i_j}$, where each $Z_j$ is a smooth linear section  of the singular locus $Z$ of  $\mathcal{P}f_4$ by either a $\mathbb{P}^{24}$ or a  $\mathbb{P}^{23}$. Let us first  conclude when $f$ is one of the $f_j$, so $Z_f$ is   one of the $Z_j$.
  We observe that $Z\subset\mathbb{P}(\bigwedge^2V_8)$ is the set of  $2$-forms of rank $\leq 4$ (the generic element of $Z$ being  of rank exactly $4$), and  has a natural birational model $\widetilde{Z}\rightarrow Z$, where
 $$\widetilde{Z}\subset G(4,V_8^*)\times \mathbb{P}(\bigwedge^2V_8),\,\,\widetilde{Z}=\{([W_4],\omega),\,W_4\subset {\rm Ker}\,\omega\}.
 $$
 The statement (a) easily follows from the above description of $\widetilde{Z}$ and the fact that the $Z_j$ are general  linear sections of codimension
 $3$ or $4$ of $Z$. Let us prove (b).
The variety  $\widetilde{Z}$ is smooth and, being a projective bundle  fibration over $G(4,V_8^*)$, has Picard rank $2$.
Its effective cone is very easy to compute: indeed, the line bundle
$l$ which is pulled-back from the Pl\"{u}cker line bundle on the  Grassmannian via the first projection $pr_1$ is clearly one extremal ray of the effective cone since the corresponding
morphism has positive dimensional fibers. There is a second extremal ray of the effective cone, which is the
class of the divisor $D$ contracted by the birational map $\widetilde{Z}\rightarrow Z$ (induced by the second projection  $pr_2$). One easily computes that this class is $2h-l$, where $h$ is the pull-back of hyperplane class on $ \mathbb{P}(\bigwedge^2V_8)$ by $pr_2$.
We observe that the fibers of $ pr_1$ are of dimension $5$, so that, when we take a general linear section of $Z$ by a codimension
$3$ or $4$ projective subspace,  getting the singular locus
$Z_j$ of $X_j$, all the properties above remain satisfied, and thus
${\rm Pic}\,\widetilde{Z}_j=\mathbb{Z} h_j+\mathbb{Z}l_j$, with effective cone generated by $l_j$ and $2h_j-l_j$.
We now finish the argument for $Z_j$: we lift our data $\{D_A\},\,D_0$ to $\widetilde{Z}_j$. We then have
$$2h_j=\widetilde{D}_0+2\widetilde{D}_A$$
in ${\rm Pic}\,\widetilde{Z}_j$,
with $\widetilde{D}_0$ effective and ${\rm dim}\,|\widetilde{D}_A|\geq n$, where $n$ is ${\rm dim}\,|h_{\mid Z_j}|$ by (a).
Let us write
$$\widetilde{D}_0=\alpha h_j+\beta l_j=  \frac{\alpha}{2} (2h_j-l_j)+(\beta+ \frac{\alpha}{2})l_j\,\,{\rm in}\,\,{\rm Pic}\,\widetilde{Z}_j,$$
with $\alpha,\,\beta \in  2\mathbb{Z}$.
Then the analysis above shows that
$$\alpha\geq 0,\,\beta+ \frac{\alpha}{2}\geq 0.$$
As $2h_j-\widetilde{D}_0$ is effective, we also have
$$2-\alpha\geq0,\,-\beta+1- \frac{\alpha}{2}\geq0.$$
As $\alpha$ is even, we only get the possibilities $\alpha=0,\,2$. If $\alpha=2$, we get $\beta=0$ and thus $2h_j-\widetilde{D}_0=0$, contradicting the assumption ${\rm dim}\,|2h-\widetilde{D}_0|\geq n$. If $\alpha=0$, we get from the  inequalities above  $\beta=0$ since $\beta$ has to be even, and $2h=2\widetilde{D}_A$, which proves
 statement (b)  (using the fact that
${\rm Pic}\,Z_j$ has no $2$-torsion and statement (a)).

We now have to prove the same result for the join
$Z_f=Z_1*Z_2*\ldots* Z_l$, which is done inductively on the number $l$, assuming, as this is satisfied in our situation, that the $Z_i$'s are simply connected. This way we are reduced to consider
only the join $Z_1*Z_2\subset \mathbb{P}^n$,  with $n=n_1+n_2+1$, of two linearly normal varieties
$Z_1\subset \mathbb{P}^{n_1},\,Z_2\subset \mathbb{P}^{n_2}$,  which satisfy the
 properties (a) and (b).
We observe  that $Z_1*Z_2$ is dominated by a $\mathbb{P}^1$-bundle over $Z_1\times Z_2$, namely
\begin{eqnarray}
\label{eqpourpinouveau} \widetilde{Z_1*Z_2}:=\mathbb{P}(\mathcal{O}_{Z_1}(1)\oplus \mathcal{O}_{Z_2}(1))\stackrel{\pi}{\rightarrow}Z_1\times Z_2,
\end{eqnarray}
 the two sections being contracted to $Z_1$, resp. $Z_2$, by the natural morphism to $Z_1*Z_2\subset \mathbb{P}^n$.
 The description (\ref{eqpourpinouveau}) of the join immediately proves (a) for $Z_1*Z_2$ once we have it for $Z_1$ and $Z_2$.
 We now turn to (b).
Let $h=\mathcal{O}_{\mathbb{P}(\mathcal{O}_{Z_1}(1)\oplus \mathcal{O}_{Z_2}(1)})$ on $\widetilde{Z_1*Z_2}$ and let
$D_0$ be a fixed effective divisor and $\{D_A\}$ be a mobile family of divisors on
$\widetilde{Z_1*Z_2}$ such that
\begin{eqnarray}\label{eqpourDzeroDA} D_0+2D_A=2h,
\\
\label{eqpourdimDA}
{\rm dim}\,\{ D_A\}\geq n.
\end{eqnarray}
 Then either $D_0$ or $D_A$ is vertical for $\pi$. Indeed, they  both restrict otherwise to a divisor of degree $\geq1$  on the fibers of $\pi$, contradicting (\ref{eqpourDzeroDA}).  Assume $D_0$ is vertical for $\pi$, that is, $D_0=\pi^{-1}(D'_0)$.
The equality
$2h-D'_0=2D_A$ says that $D'_0=2D''_0$ as divisors on $Z_1\times Z_2$ and,  as  $Z_1$ and $Z_2$ are simply connected, $D''_0\in|pr_1^*D''_{0,1}+pr_2^*D''_{0,2}|$ and both $2D''_{0,1}, \,2D''_{0,2}$ are effective.
 The divisors $D''_{0,i}$ on $Z_i$
have the property that the linear system
$|h-pr_1^*D''_{0,1}-pr_2^*D''_{0,2}|$ on $\mathbb{P}(\mathcal{O}_{Z_1}(1)\oplus \mathcal{O}_{Z_2}(1))$ has dimension
$\geq n$, which says that
$${\rm dim}\,|\mathcal{O}_{Z_1}(1)(-D''_{0,1})|+{\rm dim}\,|\mathcal{O}_{Z_2}(1)(-D''_{0,2})|\geq n_1+n_2.$$
As $2D''_{0,1}$ is effective on $Z_1$ and $2D''_{0,2}$ is effective on $Z_2$, we conclude that $D''_{0,i}=0$ and that
the $D_A$'s belong to $|\mathcal{O}_{Z_1*Z_2}(1)|$, so (b) is proved in this case.

In the case where $D_0$ is not vertical, then restricting again to the fibers of $\pi$, $D_0=2h-D'_0$ where $D'_0$ is effective and   comes from
$Z_1\times Z_2$, and $D_A$ is vertical, $D_A=\pi^{-1}(D'_A)$. Hence we have again
$D'_0=2D''_0$,  and ${\rm dim}\,|h-D'_0-D'_A|\geq n$, so the proof concludes as before that $D'_0=0$ and $D_A=0$ which contradicts
(\ref{eqpourdimDA}).
\end{proof}
We now conclude the proof of Lemma \ref{lepour4}. We have $J_f\subset I_{Z_f}$, since $Z_f$ is contained in ${\rm Sing}\,X_f$. Let
$T\subset S^2$ be a closed algebraic subset satisfying the assumptions  of  Lemma \ref{lepour4} for $f$. Then
for any $A,\,B\in T$,
$AB_{\mid Z_f}=M^2_{\mid Z_f}$, for some $M\in S^2$.
This implies that the moving  part of
${\rm div}\,A_{\mid Z_f}$ appears with multiplicity $2$.  Hence ${\rm div}\,A_{\mid Z_f}=D_0+2D_A$. We now use  Claim \ref{claim1pour4}
which implies that
the family of divisors  ${\rm div}\,A_{\mid Z_f}$ is of dimension $\geq n$, hence also the family
$\{D_A\}$ of divisors.
 We then   conclude from Claim \ref{claim2pour4}  that $T\subset \mathcal{S}q^2$.
\end{proof}

In order to conclude the proof of  Proposition \ref{procarre46}, it suffices now to make Lemma \ref{lepour4}  more precise by analyzing the schematic structure of a closed algebraic  subset $T\subset S^2$
satisfying the assumptions of this  lemma. We first  observe  the following:
\begin{lemm}\label{lecalctan}   Let $f$ be as in Lemma \ref{lepour4}  and let
$A,\,B\in S^1$ be general. Let $M:=AB$ and consider the subspace  $MS^2=MR^2_f\subset R^4_f$. Then
\begin{eqnarray}
\label{eqpourstrsch} [ MS^2:A^2]= BS^1 \subset S^2,
\end{eqnarray}

where as usual, the notation $ [ MS^2:A^2]$ is used for $\{S\in S^2,\,SA^2\in  \langle M\rangle\}$.

\end{lemm}
The lemma is obvious, using restriction to $Z_f$ and  using Claim \ref{claim2pour4}.

We now conclude the proof of  Proposition \ref{procarre46}.  We observe that $  MS^2$ is the tangent space to
$\mathcal{S}q^4$ at $M^2$, while
$BS^1\subset S^2$ is the tangent space to  $\mathcal{S}q^2$ at $B^2$.
Equation (\ref{eqpourstrsch})  thus says that a space $T\subset S^2$  satisfying the assumptions
of Lemma \ref{lepour4}  must be generically the {\it reduced} $\mathcal{S}q^2$. The conclusion of the proof then follows by a
specialization and  cycle-theoretic argument, using the fact that these  sets $T$ above are cones, hence come from closed algebraic subsets $\mathbb{P}(T)$ of $ \mathbb{P}(S^2)$.
\end{proof}
\section{Schiffer variations and Jacobian ideals \label{secschiffer}}
\begin{Defi}\label{defischi} A Schiffer variation of a homogeneous polynomial $f$ of degree $d$ in $n+1$ variables
$X_0,\ldots, \,X_n$   is a $1$-parameter family  $f+t x^d,\,t\in\mathbb{C}$, where
$x\in S^1$ is a linear form of the variables $X_0,\ldots, \,X_n$.
\end{Defi}
In the definition above, we are not interested in the linear character of the parameterization,
as this does not make sense anymore after  projection of  this line to the moduli space. We should
thus consider more generally $1$-parameter families of polynomials supported (up to the action of ${\rm GL}(n+1)$) on a
line as above.
We can also speak of  finite order Schiffer variations, which consist in looking at a finite order arc in an affine  line as above passing through  $f$.
Observe that if $g=f+ x^d$, then for any $u\in H^0(\mathbb{P}^n,T_{\mathbb{P}^n}(-1))=H^0(\mathbb{P}^n,
\mathcal{O}_{\mathbb{P}^n}(1))^*$ such that $\partial_u(x)=0$,
one has $\partial_u(f)=\partial_u(g)$. It follows that the Jacobian ideals $J_f$, resp. $J_g$
 generated by the partial derivatives of $f$, resp.  $g$,  satisfy the condition
\begin{eqnarray}\label{eqjacpetit}  {\rm dim}\, \langle J_f^{d-1},\,J_g^{d-1}\rangle \leq n+2.
\end{eqnarray}
It turns out that (\ref{eqjacpetit}) is in most cases a characterization of Schiffer variations, as shows
Proposition \ref{proschi} below.
A well-known result due (in various forms) to Carlson-Griffiths \cite{cagri}, Donagi \cite{donagi} and Mather-Yau \cite{mayau}
says the following:
\begin{prop}\label{theoyau}  Let $f,\,g$ be two homogeneous polynomials in $n+1$ variables, defining smooth hypersurfaces
in $\mathbb{P}^n$. If the Jacobian ideals $J_f^{d-1}$ and $J_g^{d-1}$ coincide, then $f$ and $g$ are in the same orbit
under the group ${\rm PGL}(n+1)$.
\end{prop}
A nice proof of this statement is given in \cite{donagi}. The example of the Fermat equation $f=\sum_iX_i^d$ and its variations $g=\sum_i \alpha_iX_i^d$ shows that
one does not always have $f= \mu g$ for some coefficient $\mu$,  under the assumptions of Proposition  \ref{theoyau}. The Mather-Yau theorem is
 the following variant which is more precise but  works only   for $d$ large enough and $f$ generic.
\begin{prop}\label{theoyauvariant}  Let $f,\,g$ be two homogeneous polynomials of degree $d$ in $n+1$ variables, defining smooth hypersurfaces
in $\mathbb{P}^n$. Assume $d\geq 4,\,n\geq 4$. If $f$ is generic and the Jacobian ideals $J_f$ and $J_g$ coincide, then $f=\mu g$ for some coefficient $\mu$.
\end{prop}
Let us  prove a
closely related statement concerning the case where $J_f^{d-1}$ and $J_g^{d-1}$ are not equal but almost equal, that is, satisfy equation (\ref{eqjacpetit}).

\begin{prop}\label{proschi}  Let $d,\,n$ be such that
\begin{eqnarray}
\label{eqsurdn} 4(n-3)+10\leq h^0(\mathbb{P}^n,\mathcal{O}_{\mathbb{P}^n}(d-2)) .
\end{eqnarray}
Then for a generic polynomial $f\in H^0(\mathbb{P}^n,\,\mathcal{O}_{\mathbb{P}^n}(d))$,
the equation (\ref{eqjacpetit}) holds if and only if
 $g$ belongs to  a Schiffer variation of $\lambda f$ for some coefficient $\lambda$.

\end{prop}
Note that (\ref{eqsurdn}) holds if $d\geq 4$ and  $n\geq 4$.
\begin{proof}[Proof of Proposition \ref{proschi}]  As $f$ is generic, its Jacobian ideal
$J_f^{d-1}$ has dimension $n+1$. The equation (\ref{eqsurdn}) can thus be written in the following form, for adequate choices of   linear coordinates
$X_0,\ldots,\,X_n$ and $Y_0,\ldots,\,Y_n$ on $\mathbb{P}^n$
\begin{eqnarray} \label{eqdifffg}
\frac{\partial g}{\partial X_i}=\frac{\partial f}{\partial Y_i}\,\,{\rm for}\,\,i=0,\ldots,\,n-1.
\end{eqnarray}
Let us now use the symmetry of partial derivatives:
$$ \frac{\partial ^2g}{\partial X_i\partial X_j}= \frac{\partial ^2g}{\partial X_j\partial X_i}.$$
Combined with (\ref{eqdifffg}), it provides,  for any $i,\,j$ between $0$ and $n-1$, the following second order equations
\begin{eqnarray}
\label{eqsecondordre}   \frac{\partial ^2f}{\partial X_i\partial Y_j}= \frac{\partial ^2f}{\partial X_j\partial Y_i}.
\end{eqnarray}
\begin{lemm} \label{lenouveau} Under the numerical assumption  (\ref{eqsurdn}), a generic polynomial $f$ of degree $d$ in $n+1$ variables does not satisfy a nontrivial  second order partial differential
equation of the type (\ref{eqsecondordre}).
\end{lemm}
\begin{proof} This is  a dimension count. The differential equations appearing in (\ref{eqsecondordre}) are linear second order equations
determined by elements $U$ of rank $\leq 4$ in ${\rm Sym}^2H^0(\mathbb{P}^n,\mathcal{O}_{\mathbb{P}^n}(1))^*$. The nonzero
elements $U$ of rank $\leq 4$ are parameterized by a variety of dimension $4(n-3)+9$. Given a nonzero $U$, the
differential equation $\partial^2_U  \phi=0$ determines a linear subspace $H_U$ of $H^0(\mathbb{P}^n,\mathcal{O}_{\mathbb{P}^n}(d))$
of codimension $h^0(\mathbb{P}^n,\mathcal{O}_{\mathbb{P}^n}(d-2))$ since
$$\partial^2_U: H^0(\mathbb{P}^n,\mathcal{O}_{\mathbb{P}^n}(d))\rightarrow H^0(\mathbb{P}^n,\mathcal{O}_{\mathbb{P}^n}(d-2))$$
is surjective.
If   (\ref{eqsurdn}) holds, the union of the spaces $H_U$ does not fill-in a Zariski open set of  $H^0(\mathbb{P}^n,\mathcal{O}_{\mathbb{P}^n}(d))$.
\end{proof}
It follows that all the equations appearing in (\ref{eqsecondordre}) are trivial, which says equivalently
that  for any $i,\,j$
$$\frac{\partial}{\partial X_i}\frac{\partial}{\partial Y_j}- \frac{\partial}{\partial X_j}\frac{\partial}{\partial Y_i}=0\,{\rm in}\,\, {\rm Sym}^2(H^0(\mathbb{P}^n,
\mathcal{O}_{\mathbb{P}^n}(1))^*).$$
These equations exactly say that
for some $\lambda\in \mathbb{C}$,
$$\frac{\partial}{\partial Y_i}=\lambda \frac{\partial}{\partial X_i}$$
for $i=0,\ldots,\,n-1$.
Thus $g-\lambda f$ satisfies  $ \frac{\partial}{\partial X_i}(g-\lambda f)=0$ for $i=0,\ldots,\,n-1$, hence
$g=\lambda f +\alpha X_n^d$ for some coefficient $\alpha$, which concludes the proof.
\end{proof}
We conclude this section by the following proposition  which shows the relevance of first order Schiffer variations to our subject.
\begin{prop}\label{proschifferrec} Let $f,\,g$ be two degree $d$  homogeneous   polynomials in $n+1$ variables defining smooth hypersurfaces $X_f$, $X_g$. Assume $d\geq 4,\,n\geq 4$ and  $f$ is generic. Then if there exists a linear isomorphism
$i: R^d_f\cong R^d_g$ mapping the set of first order Schiffer variations of $f$ to the set of first order Schiffer variations of $g$, $X_f$ is isomorphic to $X_g$.
\end{prop}
\begin{proof} The condition on $d,\, n$ are used in the following

 \begin{lemm}
 For $f$ generic with  $d\geq 4, n\geq 2$, if $x\in S^1$ is nonzero, then
  $x^d\not=0$ in $R^d_f$. Furthermore the map
  $x\mapsto x^d$ induces a (incomplete)  Veronese  imbedding  $v_f: \mathbb{P}(S^1)\hookrightarrow \mathbb{P}(R_f^d)$.
  \end{lemm}
  \begin{proof} The statement is equivalent to proving that, if $f$ is a generic homogeneous polynomial of degree $d$ in $n+1$ variables, $J_f^d\subset S^d$ does not contain any power $x^d$ of a linear form, or any sum
  $x^d-y^d$ of two such powers. In the first case, we get that $x\cdot  x^{d-1}=0$ in $R^d_f$ and in the second case, we get that $(x-y)(x^{d-1}+x^{d-2}y+\ldots+y^{d-1})=0$ in $R^d_f$. By an easy dimension count, one sees  that  for generic $f$, the multiplication map
  $x:R_f^{d-1}\rightarrow R_f^d$
  by any nonzero linear form
  $x\in S^1$ is injective, so in the first case we conclude that
  $x^{d-1}=0$ in $ R_f^{d-1}$ and in the second case, $x^{d-1}+x^{d-2}y+\ldots+y^{d-1}=0$ in $ R_f^{d-1}$, or equivalently
  \begin{eqnarray}\label{eqpourxdmoins1} x^{d-1}\in J_f^{d-1}\,\,\,{\rm or}\,\,\,x^{d-1}+x^{d-2}y+\ldots+y^{d-1}\in J_f^{d-1}.
  \end{eqnarray}
  Using the fact that $n\geq 2$, and choosing a coordinate system such that $x=X_0,\,y=X_1$, we can write (\ref{eqpourxdmoins1}) as
   \begin{eqnarray}\label{eqpourxdmoins1} X_0^{d-1}=\sum_i\alpha_i\frac{\partial f}{\partial X_i}\,\,\,{\rm or}\,\,\,X_0^{d-1}+X_0^{d-2}X_1+\ldots+X_1^{d-1}=\sum_i\alpha_i\frac{\partial f}{\partial X_i}\in J_f^{d-1}
  \end{eqnarray}
  for some nonzero coefficients $\alpha_i$.
     We thus get in both cases
  a nontrivial second order equation
  $$\sum_i\alpha_i\frac{\partial^2 f}{\partial X_2 \partial X_i}=0,$$
  which is excluded by Lemma \ref{lenouveau}.
  \end{proof}

We now conclude the proof. Using the lemma, the projectivized isomorphism $i$ induces an isomorphism $i_1:\mathbb{P}(S^1)\cong \mathbb{P}(S^1)$ between the two projected Veronese $v_f(\mathbb{P}(S^1)) \subset \mathbb{P}(R^d_f)$ and $v_g(\mathbb{P}(S^1)) \subset \mathbb{P}(R^d_g)$, that is, $i_1$ satisfies  $i\circ v_f=v_g\circ i_1$.  The projective isomorphism $i_1$ lifts to  a linear isomorphism $\tilde{i}_1:S^1\cong S^1$.
The incomplete Veronese embeddings $v_f$, resp.  $v_g$  factor canonically through  the complete Veronese embeddings
$$\mathbb{P}^n(S^1)\stackrel{V_d}{\rightarrow} \mathbb{P}(S^d)\dashrightarrow \mathbb{P}(R^d_f),$$
(resp. $\mathbb{P}^n(S^1)\stackrel{V_d}{\rightarrow} \mathbb{P}(S^d)\dashrightarrow \mathbb{P}(R^d_f)$,)
which implies that the following  diagram
\begin{eqnarray}\label{numerodiagfinfinetfin}
 \xymatrix{
&S^d\ar[r]^{\tilde{i_d}}\ar[d]& S^d\ar[d]\\
&R^d_f\ar[r]^{i}& R^d_g
,}
\end{eqnarray}
where $\tilde{i_d}:S^d\cong S^d$ is induced by $\tilde{i_1}$, is  commutative up to a scalar.

The  vertical quotient maps have for respective  kernels $J_f^d$, $J_g^d$.
We thus conclude that $\tilde{i_d}(J_f^d)=J_g^d$, and thus, by Proposition \ref{theoyau}, $X_f$ is isomorphic to $X_g$.
\end{proof}
\subsection{Formal properties of Schiffer variations\label{secschifor}}
Our strategy for the proof of Theorem \ref{theodonagivoisin}  when $d$ divides $n+1$ consists in finding
a characterization of the set of   Schiffer variations of a hypersurface $X_f$ that can be read from its local variation of Hodge structure. In fact we will need  not only  the infinitesimal variation of Hodge structure (IVHS) of $X_f$ but also the ``deformation of the IVHS'' along the Schiffer variation, which
is a higher  order argument. The IVHS itself provides  the first order invariants of the variation of Hodge structure of $X_f$ at the point $[f]$, hence part of the multiplicative structure of the Jacobian ring $R_f$, by (\ref{numerodiagcagri}). We wish in this section analyze the specifities of the first order Schiffer variations $\phi\in R^d_f$ as elements of the  Jacobian ring $R_f^*$
and  also
analyze, using (\ref{eqjacpetit}),  the way the  Jacobian ring deforms along them.

Recall that a first order Schiffer variation of $f$ is an element
$\phi=x^d\in R^d_f$, where $x\in S^1$. We will consider only
the subring $R_f^{d*} $ of $R_f^*$, because, by Theorem \ref{theCagri}, this is, when $d$ divides $n+1$, the data that we get from the IVHS of $f$. As $R_f^{d*} $ contains only the graded pieces of degree divisible by $d$, it does not
contain the linear form $x$ and by Proposition \ref{proschifferrec}, recovering the polynomial structure of $R_f^d$ precisely means recognizing the set of powers $\phi=x^d$.
However the ideal of $R_f^{*d}$ generated by such a  $\phi\in R_f^{d} $ has some special properties  that we can describe  using only
the multiplication map $R^{id}_f\otimes R^{jd}_f\rightarrow R^{(i+j)d}_f$ for $i+j\leq 3$.

 Given $\phi\in R^d_f$ and vector subspaces $I_k^{*d}\subset R^{*d}$ for $*=1,\,2,\,3$, and  $1\leq k\leq d-1$, consider  the following condition (*):
\begin{eqnarray} \label{seccond*}
  {\rm dim}\,I_{k}^{ id}={\rm dim}\,R^{id-k}_f\,\,{\rm  for }\,\,i=1,\,2,\,3.\\
 \label{item2} I_{k}^dI_{d-k}^d\subset \phi R^d_f,  \,\,I_{k}^dI_{d-k}^{2d}\subset \phi R^{2d}_f,
\\
 \label{item3} R^d_f I_{k}^{id}\subset  I_{k}^{(i+1)d}.\\
 \label{item4} I_{k}^d\cdot I_{l}^d\subset I_{k+l}^{2d}\,\,{\rm for}\,\,k+l\leq d-1.
\end{eqnarray}
Then condition (*) is satisfied by $I_{x,k}^{*d}:=x^kR^{*d-k}_f\subset R^{*d}_f$  for  $x\in S^1$   generic, with $\phi=x^d$, at least if $d$ is large enough.
Indeed,
condition (\ref{seccond*}) follows in that case from the fact that the multiplication by $x$ is injective in the relevant degrees (see Lemma \ref{claimpourplustardaussi}), at least if $d$ or $n$ are large enough, and the other conditions are obvious.

The second obvious property of a Schiffer variation  is described in  the following lemma.
\begin{lemm}\label{leconstant} Let $x\in S^1$ determine  a first order Schiffer variation $f_t=f+tx^d$ of   $f$ with tangent vector $\phi=x^d\in R^d_f$,  and let $I_{x,d-1,t}^{*d}:=x^{d-1}R^{*d-d+1}_{f_t}\subset R^{*d}_{f_t}$ be defined as above.
Then the quotient ring $R^{*d}_{f_t}/I_{x,d-1,t}^{*d}$ does not deform (as a ring) along the Schiffer variation $(f_t)$.
\end{lemm}

\begin{proof} Indeed, if $f_t=f+tx^d$, then $J_{f_t}=J_{f}$ modulo $x^{d-1}$, hence the quotient
$$S^{*d}/(J_{f_t}^{*d}+ x^{d-1}S^{(*-1)d+1})$$ is constant. A fortiori, its isomorphism class as a graded ring does not depend on $t$.
\end{proof}
\section{Proof of Theorem \ref{theodonagivoisin} when $d$ divides $n+1$}
\subsection{Specialization and Schiffer variations\label{secschideg}}
We will consider in this section  singular hypersurfaces $X_f$ of degree $d$ in  $\mathbb{P}^n$  defined by a polynomial    of the form $f=\sum_{i=1}^mf_ig_i$, with $n-2m\geq0$. If $d=2d'$ is even, we will choose the $f_i$ and $g_i$ to be of degree $d'$ and if $d=2d'+1$ we will choose the $f_i$ of degree $d'$ and the $g_i$ of degree $d'+1$.
The hypersurface $X_f$ is then singular along the variety $Z$ defined by the polynomials  $f_i$ and $g_i$ which are all of degree $\leq \frac{d+1}{2}$, and when they are generically chosen, it is  of dimension $n-2m$.
Let us start with the following result  of independent interest, which   will be important below and in the next section.
\begin{prop}\label{leconiveausing} Let $f$ be   a homogeneous polynomial of degree $d$ in $n+1$ variables, defining a hypersurface $X_f$ singular along a smooth subvariety $Z$ defined by homogeneous polynomial equations of degree
$\leq \frac{d+1}{2}$. Then     the dimension of the space $R^{k}_f$ is  equal to the  dimension of the space $R^{k}_{f_{\rm gen}}$    for a generic polynomial   $f_{\rm gen}$,  assuming
\begin{eqnarray}\label{eqestimee}k<  (n-{\rm dim}\,Z+1)\frac{d-3}{2}.
\end{eqnarray}
\end{prop}

\begin{proof}[Proof of Proposition \ref{leconiveausing}]  Recall that $J_f$ is generated by the partial derivatives $\frac{\partial f}{\partial X_i}$ for $i=0,\ldots, n$. For the generic polynomial $f_{\rm gen}$, these partial derivatives form a linear system $W$ of degree $d-1$ polynomials  with no base-point on $\mathbb{P}^n$, and the associated Koszul resolution
\begin{eqnarray}\label{eqkoszulpourJf}0\rightarrow \bigwedge^{n+1}W\otimes\mathcal{O}_{\mathbb{P}^n}(-(n+1)(d-1))\rightarrow\ldots\rightarrow W\otimes \mathcal{O}_{\mathbb{P}^n}(-(d-1)) \stackrel{\alpha}{\rightarrow} \mathcal{O}_{\mathbb{P}^n}\rightarrow 0,
\end{eqnarray}
twisted by $\mathcal{O}_{\mathbb{P}^n}(k)$, allows us to compute the dimension of $R^k_{f_{\rm gen}}$ or, equivalently, of
$J^k_{f_{\rm gen}}={\rm Im}\,H^0(\alpha(k))$, using the fact that this twisted Koszul complex, that we will denote by
$\mathcal{K}_{\mathbb{P}^n,k}^*$,  remains exact at the level of global sections, at least in strictly  negative degrees, where we put the last term $\mathcal{O}_{\mathbb{P}^n}(k)$ in degree $0$,  which is what we  need to compute ${\rm dim}\,J^k_{f_{\rm gen}}$. Indeed, we then get an equality
\begin{eqnarray}\label{eqkoszulpourJfglobal} {\rm dim}\,J^k_{f_{\rm gen}}=(n+1){\rm dim}\,S^{k-d+1}-\frac{n(n+1)}{2}{\rm dim}\,S^{k-2d+2}+\binom{n+1}{3}{\rm dim}\,S^{k-3d+3}\ldots.
\end{eqnarray}

In our special case, the partial derivatives $\frac{\partial f}{\partial X_i}$ form a linear system $W$ of degree $d-1$ polynomials  on $\mathbb{P}^n$ with base-locus $Z$ and we have to understand how this affects the computation. Clearly, the Koszul complex (\ref{eqkoszulpourJf}) is no more exact.
Let $\tau:Y\rightarrow \mathbb{P}^n$ be the  blow-up of $\mathbb{P}^n$ along $Z$, and let $E$ be the exceptional divisor of $\tau$. Then $W$ provides a base-point free  linear system, that we also denote  by $W$, of sections of the line bundle
$L:=\tau^*\mathcal{O}_{\mathbb{P}^n}(d-1)(-E)$ on $Y$. Then we have an exact Koszul complex on
$Y$ associated with $W$, which has the following form:
\begin{eqnarray}\label{eqkoszulpourJfY}0\rightarrow \bigwedge^{n+1}W\otimes\mathcal{O}_{Y}(-(n+1)L)\rightarrow\ldots\rightarrow W\otimes \mathcal{O}_{Y}(-L) \stackrel{\alpha'}{\rightarrow} \mathcal{O}_{Y}\rightarrow 0.
\end{eqnarray}
We now twist by $\tau^*\mathcal{O}_{\mathbb{P}^n}(k)$ so that ${\rm Im}\,H^0(\alpha'(k))= {\rm dim}\,J_f^k$. Let $\mathcal{K}_{Y,k}^*$ be this twisted Koszul complex.
We observe that, as only nonnegative  twists of $E$ appear in $\mathcal{K}_{Y,k}^*$, the global sections of $\mathcal{K}_{Y,k}^*$ are
$\bigwedge^i W\otimes S^{k-i(d-1)}$ in degree $-i$. So our problem is actually to prove that,
if the inequality (\ref{eqestimee}) holds, the complex   $K_{Y,k}^*$ of  global sections of $\mathcal{K}_{Y,k}^*$ is as before  exact  in strictly negative degrees. To prove this, we have to analyze
the   hypercohomology spectral sequence
\begin{eqnarray}\label{eqspecseq} E_1^{p,q}=H^q(Y,\mathcal{K}_{Y,k}^p)\Rightarrow \mathbb{H}^{p+q}(Y, \mathcal{K}_{Y,k}^*).
\end{eqnarray}
of $\mathcal{K}_{Y,k}^*$. As $\mathcal{K}_{Y,k}^*$ is exact, one has
$\mathbb{H}^{p+q}(Y, \mathcal{K}_{Y,k}^*)=0$, hence
\begin{eqnarray}\label{eqvanglobK} E_\infty^{p,q}=0.
\end{eqnarray}
  We have
$$H^q(Y,\mathcal{K}_{Y,k}^p)=  \bigwedge^{-p}W\otimes H^q(Y,\mathcal{O}_{Y}(pL(k)),$$
where
\begin{eqnarray}\label{eqformuletardivepourref} pL(k) =\tau^*\mathcal{O}_{\mathbb{P}^n}(p(d-1)+k)(-pE).
\end{eqnarray}

  We now observe that, when
$
-p<n-{\rm dim}\,Z$,
we have  $H^q(Y,\mathcal{K}_{Y,k}^p)=0$  for
  $q\not= 0,\, n$. This is because
one then has $R^i\pi_* \mathcal{K}_{Y,k}^p=0$ for $i>0$, and $$R^0\pi_* \mathcal{K}_{Y,k}^p\cong \mathcal{O}_{\mathbb{P}^n}(p(d-1)k).$$

When $
-p\geq n-{\rm dim}\,Z$,  we have  $H^q(Y,\mathcal{K}_{Y,k}^p)=0$ for $q<n-{\rm dim}\,Z-1$. Indeed, this follows again from the Leray spectral sequence
of  $pL(k)$ with respect to the map $\tau$  and the fact that $R^0\tau_*(pL(k))$ has  nonzero cohomology  only in degree
$n$, and the only nonzero other higher direct image  is  $R^{n-{\rm dim}\,Z-1}\tau_*(pL(k))$, which  contributes only to cohomology of degree $\geq n-{\rm dim}\,Z-1$.

A second source of vanishing for the $E_1^{p,q}$ of this spectral sequence comes from Kodaira vanishing applied to the line bundle $pL(k)$ on $Y$.
We observe that $Z$ is by assumption defined by equations of degree $\leq \frac{d+1}{2}$, so that
the line bundle
$\tau^*\mathcal{O}_{\mathbb{P}^n}(l)(-mE)$ on $Y$ is nef and big when $m\geq 0$ and $l>m \frac{d+1}{2}$. Using (\ref{eqformuletardivepourref}),
Kodaira vanishing thus tells us that $H^q(Y,\mathcal{K}_{Y,k}^p)=0$ for $q<n$ if
$-p(d-1)-k>-p\frac{d+1}{2}$.

Summarizing, we  proved   that  the spectral sequence (\ref{eqspecseq})  has vanishing as follows
  \begin{eqnarray}\label{eqvanepq}E_1^{p,q}=0  \,\,{\rm if}\,\,\,-p<n-{\rm dim}\,Z\,\,\,{\rm  and}\,\,\ n\not=q,
  \\
\label{eqvanepqprime}  E_1^{p,q}=0  \,\,{\rm if}\,\,\,-p\geq n-{\rm dim}\,Z\,\,\,{\rm  and}\,\,\,q< n-{\rm dim}\,Z-1,
  \\
 \label{eq3}  E_1^{p,q}=0  \,\,{\rm if}\,\,q<n\,\,{\rm and}\,\,-p(d-1)-k>-p\frac{d+1}{2}.
\end{eqnarray}

We now conclude the proof.
The complex     $K_{Y,k}^*$ of global sections of  the complex
$\mathcal{K}_{Y,k}^*$  is the complex $E_1^{*,0}$ of our spectral sequence, hence its cohomology
is the complex  $E_2^{*,0}$.  Recall that we have to prove
the vanishing of  the cohomology of $K_{Y,k}^*$ in strictly negative  degrees.
Let $a<0$ be a fixed negative integer.
There is  no nonzero differential $d_r$ with $r\geq 2$ starting from $E_2^{a,0}$ since $E_r^{p,q}=0$ for $q<0$. As
$E_\infty^{a,0}=0$  (see \ref{eqvanglobK}), it follows that, if  $E_2^{a,0}$ is nonzero, there must be a nonzero differential
$$d_r: E_r^{a-r,r-1}\rightarrow E_r^{a,0}= E_2^{a,0}.$$
Let $p=a-r$. By the vanishing statements
(\ref{eqvanepq}), (\ref{eqvanepqprime}), we must have $r-1=n$ if $-p<n-{\rm dim}\,Z$ and $r-1\geq n-{\rm dim}\,Z-1$ if  $-p\geq n-{\rm dim}\,Z$.
If $r-1=n$, then, as $a<0$, $p=a-r<-n-1$. The term  $\mathcal{K}_{Y,k}^{p}$ is then $0$.
Hence the only nontrivial differential  appears when  $r-1\not=n$.  But then,   $r\geq n-{\rm dim}\,Z$ and thus

\begin{eqnarray} \label{eqfin1} p=a-r<- n+{\rm dim}\,Z.
\end{eqnarray}
 Furthermore, using (\ref{eq3}), $-p(d-1)-k\leq -p\frac{d+1}{2}$, that is,
\begin{eqnarray}\label{eqfin2} -p\frac{d-3}{2}\leq k.
\end{eqnarray}
Combining (\ref{eqfin1}) and (\ref{eqfin2}), we proved  that the existence of a nonzero $E_2^{a,0}$ for some $a<0$ implies
\begin{eqnarray}\label{eqfin3}  (n-{\rm dim}\,Z+1)\frac{d-3}{2}\leq k.
\end{eqnarray}
which contradicts inequality (\ref{eqestimee}). Proposition \ref{leconiveausing} is thus proved.
\end{proof}

Imposing   the dimension of $Z$  to be at most $4$ (we will later choose dimension of $Z$ to be  equal to $3$ if $n$ is even and $4$ if $n$ is odd), we get
\begin{coro}\label{coroestimee}   For a polynomial $f$ as in Proposition \ref{leconiveausing} with $ d$ dividing $n+1$ and  ${\rm dim}\,Z\leq 4$, the dimensions of the spaces $R^d_f$, $R^{2d}_f$ and $R^{3d}_f$ are respectively  equal to the  dimensions of the   spaces $R^d_{f_{\rm gen}}$, $R^{2d}_{f_{\rm gen}}$ and $R^{3d}_{f_{\rm gen}}$ for  generic $f_{\rm gen}$, assuming $d\geq 13$.
\end{coro}
\begin{proof} Indeed, if    $ {\rm dim}\,Z\leq 4$ and $k\leq 3d$, (\ref{eqestimee}) is  satisfied if \begin{eqnarray}\label{eqautreestim} 3d<  (n-3)\frac{d-3}{2}.
\end{eqnarray}
As $n\geq d-1$ and $d\geq 3$, (\ref{eqautreestim}) is satisfied if $3d<  (d-4)\frac{d-3}{2}$, hence if $d\geq13$.
\end{proof}
\begin{rema}{\rm The estimate  of Corollary \ref{coroestimee} is  sharp only when  $d=n+1$.}
\end{rema}

We now assume $f=\sum_{i=1}^mf_ig_i$ with $f_i,\,g_i$ generic and $d$, $n$ are  such that the conclusion of Corollary \ref{coroestimee} holds.
We observe that, with the same notation as above, as $f$ is singular along $Z$, one has
$J_f\subset I_Z$, hence the Jacobian ring $R_f^{d*}$ has $H^0(Z,\mathcal{O}_Z(d*))$ as a quotient. We  will use the notation
$g_{\mid Z}$ for the image of an element $g$ in this quotient.  For subspaces  $I_k^{id}\subset R^{id}_f$,
let us   denote  $\overline{I}_k^{id}:=I_{k\mid Z}\subset H^0(Z,\mathcal{O}_Z(id))$.
Let us prove the following.
\begin{lemm}\label{leidmoinsun}  Let  $\phi\in R_f^{d}$  and  $I_k^{*d}\subset R_f^{*d}$  satisfy condition (*)   (see (\ref{seccond*})-(\ref{item4})).
Then either  $ \overline{I}_{d-1}^d=0$   and $ \overline{I}_{d-1}^{2d}=0$ or  there exists an  element  $g$ of $R^{k}_f$ with $k\geq d-1$, such that $g_{\mid Z}\not=0$ and
\begin{eqnarray}\label{eqpourIdmoinsun} \overline{ I}_{d-1}^d \subset g H^0(Z,\mathcal{O}_Z(d-k)),\,\,\overline{I}_{d-1}^{2d} \subset g H^0(Z,\mathcal{O}_Z(2d-k)).
\end{eqnarray}
\end{lemm}

\begin{proof}  For a nonzero  linear system $W$  on $Z$, let us denote by ${\rm FL}(W)$ (for the ``fixed locus'') the divisorial part of the base-locus of $W$. We now observe that, as $Z$ is a smooth  complete intersection  of dimension at least $3$, one has ${\rm Pic}\,Z=\mathbb{Z}\mathcal{O}_Z(1)$  by Grothendieck-Lefschetz theorem. In particular, if $\overline{I}_l^k\not=0$,  we have
$${\rm FL}(\overline{I}_l^k)=D_l\in |\mathcal{O}_Z(d_{l,k})|,$$
for some nonnegative integers $d_{l,k}$.

We first make  the following
\begin{claim} \label{claimpourI1} For $d$ large enough, one has $\overline{I}_1^d\not=0$, and $d_{1,d}\leq 1$.
\end{claim}
\begin{proof} This is proved by a dimension argument. Indeed, it suffices to prove that
\begin{eqnarray}\label{eqineqpourdimiun}{\rm dim}\,\overline{I}_1^d> h^0(Z,\mathcal{O}_Z(d-2)).
\end{eqnarray}
As ${\rm dim}\,I_1^d={\rm dim}\,R^{d-1}_f$ by (\ref{seccond*}), one has ${\rm dim}\,\overline{I}_1^d\geq {\rm dim}\,S^{d-1}-{\rm dim}\,I_Z(d)
$ and it thus  suffices to prove that \begin{eqnarray}\label{eqpourI1bar}
h^0(Z,\mathcal{O}_Z(d-2))<{\rm dim}\,S^{d-1}-{\rm dim}\,I_Z(d).\end{eqnarray}
Recalling that $Z$ is a complete intersection of $2m<n$ hypersurfaces defined by equations $f_i$ of degree $d'$ and
$g_i$ of degree  $d''$, with $d''-1\leq d'\leq d''$, and $d'+d''=d$, we conclude that
$$h^0(Z,\mathcal{O}_Z(d-2))={\rm dim}\,S^{d-2}-m({\rm dim}\,S^{d''-2}+{\rm dim}\,S^{d'-2}),$$
and that
$${\rm dim}\,I_Z(d)\leq m({\rm dim}\,S^{d''}+{\rm dim}\,S^{d'}).$$
Inequality (\ref{eqpourI1bar}) will thus be a consequence of
\begin{eqnarray}\label{eqpourI1bar1} {\rm dim}\,S^{d-2}-m({\rm dim}\,S^{d''-2}+{\rm dim}\,S^{d'-2})<{\rm dim}\,S^{d-1}-m({\rm dim}\,S^{d''}+{\rm dim}\,S^{d'}).
\end{eqnarray}
Inequality (\ref{eqpourI1bar1}) easily follows (at least for $d$ large enough) from our conditions $n> 2m$ and $d'=d'' =d/2$ if $d$ is even, $d'=d''_1 =(d-1)/2$ if $d$ is odd.
\end{proof}
We now use  the fact that $\overline{I}_1^d\overline{I}_{d-1}^d\subset \phi_{\mid Z}\cdot H^0(Z,\mathcal{O}_Z(d))$
(see (\ref{item2})). Combined with Claim \ref{claimpourI1}, this implies
that, either  $\overline{I}_{d-1}=0$, or $d_{d-1,d}\geq d-1$. Similarly,
as $$\overline{I}_1^{d}\overline{I}_{d-1}^{2d}\subset \phi_{\mid Z}\cdot H^0(Z,\mathcal{O}_Z(2d)),$$ we conclude that
$d_{d-1,2d}=:k\geq d-1$. Finally we use the fact that  $H^0(Z,\mathcal{O}_Z(d))\cdot  \overline{I}_{d-1}^{d}\subset  \overline{I}_{d-1}^{2d}$ (see (\ref{item3})) to deduce that the same $g$ of degree $k\geq d-1$ works for both $\overline{I}_{d-1}^{d}$ and $ \overline{I}_{d-1}^{2d}$. Lemma \ref{leidmoinsun} is now proved.
\end{proof}
We now want to study, for a generic  polynomial $f$ of the form $\sum_{i=1}^mf_ig_i$ as above, the elements $\phi\in R^d_f$ which both satisfy  condition (*)    and the property described in the assertion of Lemma \ref{leconstant}.
As Lemma \ref{leconiveausing}     and Corollary \ref{coroestimee} hold only for $R^d_f$, $R^{2d}_f$ and $R^{3d}_f$ and not for the whole
$R^{*d}_f$ when $f$ is singular, we are going to use only the data of the multiplication map
of $R^*_f$ in degree $d$, which is described by a triple
$(R^d_f,R^{2d}_f,\mu)$ consisting of (isomorphism class of) two vector spaces of dimensions ${\rm dim}\,R^d_{f_{\rm gen}}$, resp.
 ${\rm dim}\,R^{2d}_{f_{\rm gen}}$, and  a symmetric linear map
 $$\mu_f: R^d_f\otimes R^d_f\rightarrow R^{2d}_f.$$
We will also consider similar data $\overline{\mu}_f: \overline{R}^d\otimes \overline{R}^d\rightarrow \overline{R}^{2d}$ for quotients of
$R_f^{d*}$ and $\mu:S^d\otimes S^d\rightarrow S^{2d}$ for the multiplication in the polynomial ring itself. We will call such data a ``partial ring''.

We study now elements
$\phi\in R^d_f$  satisfying the following condition (**) (satisfied by Schiffer variations, see Section \ref{secschifor})

\vspace{0.5cm}

(**)  {\it (i) For $1\leq k\leq d-1$, there exist vector subspaces $I_{k}^d\subset R^d_f$, $I_{k}^{2d}\subset R^{2d}_f$, $I_{k}^{3d}\subset R^{3d}_f$ satisfying  condition (*)    (see (\ref{seccond*})-(\ref{item4})).

(ii)  Along a $1$-parameter family
$f_t$, with $f_0=f$ and $\frac{d}{dt}(f_t)_{|t=0}=\phi$, there exist data
$I_{k,t}^{*d}\subset  R_{f_t}^{*d}$, $*=1,\,2,\,3$, $k=1,\ldots, d-1$, associated  to $\phi_t=\frac{d}{dt}(f_t)\in R^d_{f_t}$, and  also satisfying condition (*).

(iii) The (isomorphism class of the)  partial ring  $(R^d_{f_t}/I_{d-1,t}^d, R^{2d}_{f_t}/I_{d-1,t}^{2d},\overline{\mu})$  does not deform  with $t$.}
\begin{prop}\label{leschiffer} For  $d$  sufficiently large and for a generic $f=\sum_{i=1}^mf_ig_i$  as above, any  $\phi\in R^d_f$ satisfying
(**) is a  first order Schiffer variations of $f$.
\end{prop}
\begin{rema} {\rm We will also prove later on (see Lemma \ref{lecomplement}) that, in the situation of Proposition \ref{leschiffer}, for a generic first order  Schiffer variation $\phi=x^d$, the only spaces $I_k^{id}$ satisfying Condition (*) with the given $\phi$  are the spaces  $I_{x,k}^{id}=x^kR^{id-k}_f$, hence are  determined by $\phi$.}
\end{rema}
 The proof of Proposition \ref{leschiffer}  will use  several preliminary lemmas.
\begin{lemm} \label{leintermediaire} The  assumptions being the same as in Proposition
\ref{leschiffer}, then

 (a) If $\overline{I}_{d-1}^{2d}=0$, $f_t$ remains singular along $Z$ (or rather, a subvariety deduced from $Z$ by the action of an automorphism of ${\rm PGL}(n+1)$). In particular
$\phi_{\mid Z}=0$.

(b)  If ${\rm FL}(\overline{I}_{d-1}^{2d})$ is defined by
$g\in H^0(Z,\mathcal{O}_Z(k))$, $f_t$ remains (modulo the action of ${\rm PGL}(n+1)$) singular along the  locus  $Z_g:=\{g=0\}\subset Z$.

(c) If $k\geq d$ in (b), $f_t$ remains (modulo the action of ${\rm PGL}(n+1)$) singular along $Z$.

\end{lemm}
\begin{proof} (a) If $\overline{I}_{d-1}^{2d}=0$, the partial ring $(R^d_{f}/I_{d-1}^d, R^{2d}_{f}/I_{d-1}^{2d},\overline{\mu}_f)$ admits  the partial ring $(H^0(Z,\mathcal{O}_Z(d)), H^0(Z,\mathcal{O}_Z(2d)),\mu_Z)$ as a quotient. As by assumption, the quotient $$(R^d_{f_t}/I_{d-1,t}^d, R^{2d}_{f_t}/I_{d-1,t}^{2d},\overline{\mu}_{f_t})$$ of
$(R^d_{f_t}, R^{2d}_{f_t},{\mu}_{f_t})$ is isomorphic to  $(R^d_f/I_{d-1}^d, R^{2d}_f/I_{d-1}^{2d},\overline{\mu})$, we conclude that the partial ring $(R^d_{f_t},R^{2d}_{f_t},\mu_{f_t})$ also admits
the partial ring  $(H^0(Z,\mathcal{O}_Z(d)), H^0(Z,\mathcal{O}_Z(2d)),\mu_Z)$ as a quotient.
Denoting by $\alpha_t: S^{*d}\rightarrow H^0(Z,\mathcal{O}_Z(*d))$  the quotient map for $*=1,\,2$,  this means  that we have a commutative diagram
\begin{eqnarray}\label{numerodiagpourZft}
 \xymatrix{
&S^d\otimes S^d\ar[d]^{\alpha_t\otimes\alpha_t}&\stackrel{\mu}{\xrightarrow{\hspace*{2.5cm}}}& S^{2d}\ar[d]^{\alpha_t}\\
&H^0(Z,\mathcal{O}_Z(d))\otimes H^0(Z,\mathcal{O}_Z(d))&\stackrel{\mu_Z}{\xrightarrow{\hspace*{1.5cm}}}& H^0(Z,\mathcal{O}_Z(2d)),}
\end{eqnarray}
where  $\alpha_t$ is surjective with kernel containing $J_{f_t}$, since it factors through $R_{f_t}$. The map $\alpha_t$ gives an embedding $j_t$ of
$Z$ in $\mathbb{P}((S^d)^*)$. As the quadrics in ${\rm Ker}\,\mu$ are the defining equations for the $d$-th Veronese
embedding $V_d(\mathbb{P}^n)$ in $\mathbb{P}((S^d)^*)$, one concludes that $j_t$ factors through an embedding
$j'_t$ of $Z$ in $\mathbb{P}^n=\mathbb{P}((S^1)^*)$, that is $j_t=V_d\circ j'_t$. As  $Z\stackrel{j'_0}{\hookrightarrow}\mathbb{P}^n$ is the natural embedding of a complete intersection in $\mathbb{P}^n$ of dimension $>0$, the small deformations of the morphism  $j'_0: Z\rightarrow\mathbb{P}^n$ are induced by the action of ${\rm PGL}(n+1)$. Hence  for $t$ close to $0$, $j'_t$ is, up to the action of ${\rm PGL}(n+1)$,  the original  embedding.
Finally, as the map $\alpha_t=(j'_t)^*$ contains $J_{f_t}$ in its kernel,
$J_{f_t}$ vanishes on $j'_t(Z)$, which means that $f_t$ is singular along $j'_t(Z)$.

(b) We know that for $t=0$, and $*=1,\,2$,  $(I_{d-1}^{*d})_{\mid Z}$ is contained in the ideal generated by $g$.
It follows that the partial ring $(R^d_{f}/I_{d-1}^d, R^{2d}_{f}/I_{d-1}^{2d},\overline{\mu}_f)$ has the partial ring $$(H^0(Z_g,\mathcal{O}_{Z_g}(d)), H^0(Z_g,\mathcal{O}_{Z_g}(2d)),\mu_{Z_g})$$ as a quotient, where
$Z_g:=\{g=0\}\subset Z$.
We can then argue exactly as before, using the fact that $Z_g$ is a complete intersection of strictly positive dimension  in $\mathbb{P}^n$. We then conclude that $f_t$ is singular along $j'_t(Z_g)\subset \mathbb{P}^n$ and that the embedding $j'_t$ of $Z_g$ in $\mathbb{P}^n$  is deduced from $j'_0$ by the action of an element of ${\rm PGL}(n+1)$.

(c)  By (b), we know  that, modulo the action of ${\rm PGL}(n+1)$,  $f_t$ remains singular along the hypersurface $\{g=0\}$ in $Z$. As the singular locus of $f_t$ is defined by the partial derivatives of $f_t$ which are  degree $d-1$ polynomials,  and   $Z$ is smooth connected, we conclude, when $k={\rm deg}\,g\geq d$, that the partial derivatives of $f_t$ vanish along $Z$, which proves (c).
\end{proof}
We next make the following observation
\begin{lemm}\label{lemmapuissance}  Let $Z$ be a smooth complete intersection of dimension $\geq3$
of hypersurfaces $X_{h_j}$ of degrees $d_j\geq 2$ and let $f_t$ be a polynomial of degree $d$  such that
$f_t$ is  singular along $Z_g$ for some $0\not=g\in H^0(\mathcal{O}_Z(k))$ with $k\geq d-1$. Then  either $f_t$ is singular along $Z$ or there exists an element $x\in S^1$ such that $g=x^{d-1}_{\mid Z}$ and  $f_t-\alpha_t x^d$ is singular along  $Z$ for some scalar $\alpha_t$.
\end{lemm}
\begin{proof} We first claim  that if $f_{t\mid Z}=0$, then $f_t$ is singular along $Z$. This is proved as follows:
As $f_{t\mid Z}=0$, we can write $f_t=\sum_ja_j h_j$, with ${\rm deg}\,a_j=d-d_j$.
As $Z$ is smooth, the differential of $f_t$ vanishes at a point $z\in Z$ if and only if all $a_j$ vanish at $z$. As the $a_j$'s are of degree
$<d-1$ and ${\rm deg}\,g\geq d-1$, the vanishing of $df_t$ along $Z_g$ implies the vanishing of $df_t$ along $Z$.

Next, if $k\geq d$,  we conclude that the partial derivatives of $f$ vanish  identically along $Z$, since they vanish along $Z_g$,  so $f$ is singular along $Z$. We can thus assume that $f_{\mid Z}\not=0$ and $k=d-1$.

We then  claim  that  there exists an $x\in S^1$ and a scalar $\alpha_t$  such that $(f_t-\alpha_tx^d)_{\mid Z}=0$ and $g= x^{d-1}_{\mid Z}$.
We use here the fact that ${\rm dim}\,Z\geq3$ so that ${\rm Pic}\,Z=\mathbb{Z}\mathcal{O}_Z(1)$. We decompose
$g\in H^0(Z,\mathcal{O}_Z(d-1))$ into irreducible factors as
$$g=\prod_j \gamma_j^{a_j},$$
where $\gamma_j\in H^0(Z,\mathcal{O}_Z(d_j))$ and $\sum_ja_jd_j=d-1$. Now if $f_{t\mid Z}$ vanishes to order $b_j$ along
$\{\gamma_j=0\}$, $df_{t}$ vanishes to order $\leq b_j-1$ along
$\{\gamma_j=0\}$. We thus conclude that $b_j\geq a_j+1$.
As $\sum_jb_j\leq d$ and $\sum_ja_j=d-1$, we conclude that there is a single $j$ and the corresponding  $a_j$ equals $d-1$, which proves that $g= x^{d-1}_{\mid Z}$ for some $x\in S^1$. It follows that $f_{\mid Z}$ vanishes along $ x^{d-1}_{\mid Z}=0$  and the fact that the derivatives of $f$ also vanish along $ x^{d-1}_{\mid Z}=0$ implies that $f_{\mid Z}$ is  proportional to $x^d_{\mid Z}$, proving
 the second claim.

The second claim finally implies   Lemma \ref{lemmapuissance} since  $f_t-\alpha_t x^d$  vanishes along $Z$ and is also  singular along $Z_g$,
with  $g=x^{d-1}_{\mid Z}$,   so that the first claim applies to show that $f_t-\alpha_tx^d$ is   singular along $Z$.
\end{proof}
\begin{proof}[Proof of Proposition \ref{leschiffer}] With the notation and assumptions of  Proposition \ref{leschiffer}, Lemma
\ref{leintermediaire} tells us that, modulo the action of ${\rm GL}(n+1)$, we can assume  $f_t$ is singular along $Z$ or $f_t$ is singular along $Z_g$. Lemma
\ref{lemmapuissance} then says that,  for some  $x\in S^1$, $g=x^{d-1}_{\mid Z}$ and
$f_t-\alpha_t x^d$ is singular along $Z$ for any $t$, and the same is   true for $\phi=\frac{\partial f_t}{\partial t}_{\mid t=0}$.
It follows that either (i)  $\phi\in I_Z^2(d)$ or  (ii) $\phi-x^d\in I_Z^2(d)$.

We use now the fact (this is (\ref{item2}) in   condition (*)) that \begin{eqnarray}\label{eqfinfinprop}  I_k^d\cdot I_{d-k}^{d}\subset \phi R^d_f,
\end{eqnarray} for $1\leq i\leq d$,  with $I_k^d\subset R^d_f$ of dimension equal to
${\rm dim}\,R^{d-k}_f$ (this is (\ref{seccond*}) in  condition (*)).
 We previously used this condition only for $k=1$. We are going to use it for $k=3$
to  prove the following claim  which excludes case (i).
\begin{claim}\label{claimpourI1etIdmoins1} For $d$ large enough  and $f,\,Z$  generic, a nonzero element  $\phi\in R^d_f$ satisfying  condition (*) for adequate spaces $I_k^{*d}\subset R^{*d}_f$
cannot belong to $I_Z^2(d)$.
\end{claim}
\begin{proof} We argue by contradiction and assume that $\phi\in I_Z^2(d)$  and its image in $R^d_f$ satisfies condition (*). First of all, we use the same dimension arguments as in the proof of Claim \ref{claimpourI1} to show that
 $\overline{I}^d_3:=({I}_3^d)_{\mid Z}\not=0$. More precisely, we can show that it is of dimension $>h^0(Z,\mathcal{O}_Z(d-4))$, at least if $d$ is large enough.
As  $\phi\in I_Z^2(d)$, we have in particular  $\phi _{\mid Z}=0$, and thus  $I_{d-3}^d\subset I_Z(d)$ since
\begin{eqnarray}\label{eqnouve12avril} I_{3}^d\cdot I_{d-3}^{d}\subset \phi R^d_f \subset I_Z(2d)\,\,{\rm mod}\,\,J_f.
\end{eqnarray}
On the one hand, as ${\rm dim}\,I_{d-3}^d={\rm dim}\,S^3$ for $d>4$, and ${\rm dim}\,I_Z^2(d)<{\rm dim}\,S^3$, $I_{d-3}^d$ is not contained in $I_Z^2(d)$.
On the other hand, if we look at
the image of $I_{d-3}^d$ in $I_Z(d)/(I_Z^2(d)+J_f^d)$, it is  annihilated by multiplication by elements of
$\overline{I}_3^d$ acting by
$$H^0(Z,\mathcal{O}_Z(d))\supseteq\overline{I}_3^d\ni \alpha: I_Z(d)/(I_Z^2(d)+J_f^d)\rightarrow I_Z(2d)/(I_Z^2(2d)+J_f^{2d}).$$ This follows indeed from the condition that $\phi\in I_Z^2(d)$ and (\ref{eqnouve12avril}).
Now, writing $f=\sum_jf_jg_j$ with ${\rm deg}\,f_j=d'$ and ${\rm deg}\,g_j=d''$ , we have a graded  isomorphism (given by differentiation along
$Z$)
$$ (I_Z/I_Z^2)(*)\cong \oplus_{j=1}^m H^0(Z,\mathcal{O}_Z(*-d'))\bigoplus \oplus_{j=1}^mH^0(Z,\mathcal{O}_Z(*-d'')),$$
which to $\sum_ja_jf_j+b_jg_j$ associates $(a_{j|Z},b_{j|Z})_{j=1,\ldots,m}$.
By the Leibniz rule, this isomorphism maps
$\frac{\partial f}{\partial X_i}\in J_f$ to the $2m$-uple $(\frac{\partial g_j}{\partial X_i},\frac{\partial f_j}{\partial X_i})_{j=1,\ldots,m}$. In other words, observing that
we have a natural isomorphism $N_{Z/\mathbb{P}^n}\cong N_{Z/\mathbb{P}^n}^*(d)$ given by the  quadratic form
defined as the Hessian of $f$ along $Z$,
we have on the one hand the composite morphism
$\mathcal{I}_Z\rightarrow N_Z^*\cong N_Z(-d)$
and on the other hand  the normal bundle sequence of
$Z$
\begin{eqnarray} \label{eqnormalbunlde}0\rightarrow T_Z\rightarrow T_{\mathbb{P}^n\mid Z}\stackrel{\beta}{\rightarrow } N_Z\rightarrow  0.
\end{eqnarray}
Then the computation above shows that
\begin{eqnarray}\label{eqisopourIZIZcarre} I_Z(*)/(I_Z^2(*)+J_f^*)\cong H^0(Z,N_Z(*-d))/{\rm Im}\,H^0(\beta(*-d)),
\end{eqnarray}
and these isomorphisms are compatible with the multiplication map
by   $b\in H^0(Z,\mathcal{O}_Z(d))$. Finally, the exact sequence (\ref{eqnormalbunlde}) together with the fact that ${\rm dim}\,Z\geq 3$ show that
the right hand side in (\ref{eqisopourIZIZcarre}) is isomorphic to $H^1(Z,T_Z)$ for $*=d$.
Let now $w\in I_{d-3}^d\subset I_Z(d)$ such  that $w\not=0$ in $I_Z(d)/(J_f^d+I_Z^2(d))$.
Then $w$ has a nonzero image $\overline{w}\in H^1(Z,T_Z)$ and
$\overline{w}$ is annihilated by multiplication by any $b\in \overline{I}_3^{d}\subset H^0(Z,\mathcal{O}_Z(d))$, that is,
\begin{eqnarray}\label{eqiuiditquebwegal0} b\overline{w}=0\,\,{\rm in}\,\,H^1(Z,T_{Z}(d))
\end{eqnarray}
for any $b\in \overline{I}_3^{d}$.
The extension class $\overline{w}\in H^1(Z,T_{Z})$ determines a vector bundle $F$ on $Z$ which fits in an exact sequence
\begin{eqnarray} 0\rightarrow T_Z\rightarrow F \rightarrow \mathcal{O}_Z\rightarrow 0,
\label{eqexdefF}
\end{eqnarray}
and the condition (\ref{eqiuiditquebwegal0}) says equivalently that
$\overline{I}_3^{d}\subset H^0(Z,\mathcal{O}_Z(d))$ lifts to sections of $F(d)$.
Let $\mathcal{G}\subset F(d)$ be the coherent subsheaf generated by the global sections
of $F(d)$. Observe that ${\rm det}\,\mathcal{G}=\mathcal{O}_Z(k)$ with $k\geq 0$ since ${\rm Pic}\,Z=\mathbb{Z}\mathcal{O}_Z(1)$ and $\mathcal{G}$ is generated by its sections.
Assume first that $\mathcal{G}$ has rank $1$. Then we have
$$h^0(Z,\mathcal{O}_Z(k))\geq{\rm dim}\,\overline{I}_3^{d}$$
and we already noted that the right hand side is $>h^0(Z,\mathcal{O}_Z(d-4))$. It follows that
$k\geq d-3$, and that for some $0\not=\sigma\in h^0(Z,\mathcal{O}_Z(3))$, one has
\begin{eqnarray}\label{eqdiffeavecdeg1} \sigma\overline{w}=0\,\,{\rm in}\,\,H^1(Z,T_Z(3)).
\end{eqnarray}
Equation (\ref{eqdiffeavecdeg1}) says that $\overline{w}$ is coming from a section of $H^0(Z_\sigma,T_{Z\mid Z_\sigma}(3))$, where
$Z_\sigma:=\{\sigma=0\}\subset Z$.
A dimension count shows that for $d$ large enough and $Z$ generic as above, there does not exist
a cubic section $Z_\sigma$ of $Z$ and a nonzero section of $T_{Z\mid Z_\sigma}(3)$. This case is thus ruled-out.
We thus conclude that the rank of $\mathcal{G}$ is at least $2$. We then get a contradiction as follows.
Let now $\mathcal{G}':={\rm Ker}\,\alpha: (\mathcal{G}\rightarrow \mathcal{O}_Z(d))$, where the morphism $\alpha$ is the restriction to $\mathcal{G}$ of the morphism $F(d)\rightarrow \mathcal{O}_Z(d)$ deduced from the exact sequence (\ref{eqexdefF}). By this exact sequence, $\mathcal{G'}$ is  a subsheaf of $T_Z(d)$ and we
 have ${\rm det}\,\mathcal{G}'=\mathcal{O}_Z(k')$ with $k'\geq -d$. Thus the slope of $\mathcal{G}'$ is at least
$-d$.
Recall that $Z$ is a complete intersection of $m$ hypersurfaces of degree $d'$ and $m$ hypersurfaces of degree $d''$ with $d'+d''=d$ and that $s:={\rm dim}\,Z$ is equal to $3$ or $4$. It follows  that $n=2m+s$ and
$$K_Z=\mathcal{O}_Z(-n-1+md)= \mathcal{O}_Z(-2m-s-1+md)= \mathcal{O}_Z(m(d-2)-s-1).$$
It follows that ${\rm det}\,T_Z(d)=\mathcal{O}_Z(-m(d-2)+s+1+sd)$ and for $-m(d-2)+s+1+sd<0$,
 the slope of $T_Z(d)$ is thus at most $\frac{-m(d-2)+s+1+sd}{4}\leq \frac{-m(d-2)+5+4d}{4}$.
Hence we have
$${\rm slope}\,\mathcal{G'}>{\rm slope}\,T_Z(d) \,\,\,{\rm
if }\,\,\, -d>\frac{-m(d-2)+5+4d}{4},$$
which holds if  $m\geq 10,\,d\geq13$.
This gives  a contradiction for $d$ large enough  since $Z$ is a variety with ample canonical bundle, hence has stable tangent bundle by \cite{bogomolov},
\cite{yau} or \cite{tsuji}.
The claim is thus proved.
\end{proof}

We are thus in  case (ii), that is,
\begin{eqnarray}
\label{eqphixd} \phi=x^d+\alpha\,\,{\rm  mod }\,\,J_f^d
\end{eqnarray}
for some  $\alpha\in  I_Z^2(d)$,
 and    we need  to show  that, in fact,
$\phi= x^d$ mod $J_f^d$.
We start with  the following lemma, where we use again the notation $\overline{I}_{k}^d:=(I_{k}^d)_{\mid Z}$.
\begin{lemm}\label{lepourproschIun} One has $\overline{I}_{1}^d\subset xH^0(Z,\mathcal{O}_Z(d-1))$.
\end{lemm}

\begin{proof} As $\phi_{\mid Z}=x^d_{\mid Z}$, we have, by equations (\ref{item2}) and (\ref{item4})  of condition (*) followed by restriction to $Z$,
\begin{eqnarray}\label{eqschi1} \overline{I}_{1}^d\cdot  \overline{I}_{d-1}^d\subset x^dH^0(Z,\mathcal{O}_Z(d)),\\
\overline{I}_{1}^d\cdot  \overline{I}_{d-1}^{2d}\subset  x^dH^0(Z,\mathcal{O}_Z(2d))
\nonumber
\end{eqnarray}
 If $\overline{I}_{1}^d\not\subset xH^0(Z,\mathcal{O}_Z(d-1))$, then (\ref{eqschi1}) imply that
\begin{eqnarray}\label{eqschi2} \overline{I}_{d-1}^d\subset\mathbb{C} x^d,\,\overline{I}_{d-1}^{2d}\subset x^dH^0(Z,\mathcal{O}_Z(d)).
\end{eqnarray}
By  Lemma \ref{leintermediaire}, (c),  this implies
that $f_t$ remains singular along $Z$.
Thus $f_t\in I_Z^2(d)$ and $\phi\in I_Z^2(d)$, contradicting (\ref{eqphixd}).
\end{proof}
\begin{coro}\label{coroIprime} Let  $I'_1\subset I_1^d$ be defined
by
$ I'_1=I_1^d\cap xR^{d-1}_f
$,
and let $\overline{I}'_1:=(I'_1)_{\mid Z}\subset xH^0(Z,\mathcal{O}_Z(d-1))$. Then for $d$ (hence also $n$) large enough
$${\rm dim}\,\overline{I}'_1>h^0(Z,\mathcal{O}_Z(d-2)).$$
\end{coro}
\begin{proof} Indeed, as $\overline{I}^d_1\subset xH^0(Z,\mathcal{O}_Z(d-1))$, we have
$I^d_1\subset xR^{d-1}_f+I_Z(d)$, hence
$${\rm codim}\,(I'_1\subset I_1^d)\leq {\rm dim}\,I_Z(d),$$
which implies a fortiori
$${\rm codim}\,(\overline{I}'_1\subset \overline{I}_1^d)\leq {\rm dim}\,I_Z(d).$$
Using the fact that ${\rm dim}\,I_1^d={\rm dim}\,R^{d-1}_f$ (see (\ref{seccond*}) in condition (*)), the inequality  ${\rm dim}\,\overline{I}'_1>h^0(Z,\mathcal{O}_Z(d-2))$ is then proved for $d$ large enough in the same way
as the inequality
(\ref{eqineqpourdimiun}) proved in Claim \ref{claimpourI1}.
\end{proof}

We come back to our  $\phi=x^d+\alpha$ satisfying condition (**), with $0\not=x\in S^1$, and $\alpha\in I_Z^2(d)$.
By (\ref{eqschi1}) and using the fact that ${\rm dim}\,\overline{I}_1^d> h^0(Z,\mathcal{O}_Z(d-2))$, we conclude that
$$\overline{I}_{d-1}^d\subset x^{d-1}H^0(Z,\mathcal{O}_Z(1)),$$
so that we can write, for any $w\in I_{d-1}^d$, $w=x^{d-1}y+k_y$,
where $y\in S^1$ and $k_y\in I_Z(d)$ mod. $J_f^d$.
For $a=xa'\in I_1'\subset I_{1}^d\subset R^d_f$, and $w\in I_{d-1}^d$, we then  have by equations (\ref{item2}) and (\ref{item4}) in condition (*) and recalling that   $\phi=x^d+\alpha$,
\begin{eqnarray}\label{equnpeumodifiee} xa'w=xa'(x^{d-1}y+k_y)=(x^d+\alpha)\gamma_{a',w}\,\,{\rm in}\,\,R^{2d}_f.
\end{eqnarray}
Restricting to $Z$, and using the fact that  $k_y\in I_Z(d),\,\alpha\in I_Z^2(d)$, we get  $(\gamma_{a',w})_{\mid Z}=a'_{\mid Z}y_{\mid Z}$, which we write
$$\gamma_{a',w}=a'y+\gamma'_{a',w}$$
for some $\gamma'_{a',w}\in I_Z(d)$ which depends linearly on $a'$, for $w$ fixed.

We now use again the observation  that ${\rm dim}\,I_Z(d)$ is (asymptotically) small compared to
$h^0(Z,\mathcal{O}_Z(d-2))$ and conclude that for $a'$ in  a subspace
$I_1''\subset I'_1$  such that ${\rm dim}\,(I''_1)_{\mid Z}> h^0(Z,\mathcal{O}_Z(d-2))$,
one can take $\gamma'_{a',w}=0$ in $R^d_f$, so that
(\ref{equnpeumodifiee}) becomes $$xa'(x^{d-1}y+k_y)=(x^d+\alpha)a'y\,\,{\rm in}\,\,R^{2d}_f,$$ that is,
\begin{eqnarray}\label{equnpeumodifiee11} xa' k_y=\alpha a'y \,\,{\rm in}\,\,R^{2d}_f.
\end{eqnarray}
The right hand side belongs to $(I_{Z}^2(2d)+J_f^{2d})/J_f^{2d}$. We argue now as in the proof of
Claim \ref{claimpourI1etIdmoins1} to deduce that $k_y\in (I_{Z}^2(d)+J_f^{d})/J_f^{d}$. Indeed, we consider the image
$\overline{k}_y$  of
$k_y$ in $I_Z(d)/(I_Z^2+J_f^d)$ and (\ref{equnpeumodifiee11}) says that it is annihilated by
multiplication by $xa'$ for $xa'\in I_1'$, which is of large dimension. Then we conclude that
$\overline{k}_y=0$.

The equations (\ref{equnpeumodifiee11}) are thus relations in $(I_Z^2+J_f)/J_f$. We claim
that
\begin{eqnarray} \label{eqinters}I_Z^2(2d)\cap J_f^{2d}=I_Z(d+1)\cdot J_f^{d-1}.\end{eqnarray}
Indeed, recall from the proof of Claim \ref{claimpourI1etIdmoins1} that
the image of $J_f^*$ in $I_Z(*)/I_Z^2$ identifies naturally with  the image of
$H^0(\mathbb{P}^n, T_{\mathbb{P}^n}(*-d+1))$ in $H^0(Z,N_Z(*-d+1))$.
We have the exact sequence
$$0\rightarrow T_Z\rightarrow T_{\mathbb{P}^n\mid Z}\rightarrow N_Z\rightarrow 0$$
and we observe as in the proof of Claim \ref{claimpourI1etIdmoins1} that the stability of the tangent bundle of $Z$ implies that
$h^0(Z,T_Z(d))=0$ for $d$ (hence $n$) large enough.
It follows that the map
$H^0(Z,T_{\mathbb{P}^n\mid Z}(d))\rightarrow H^0(Z,N_Z(d))$ is injective, and thus
$$I_Z^2(2d)\cap J_f^{2d}={\rm Ker}\,(J_f^{2d}\rightarrow I_Z/I_Z^2(2d))$$
comes from $H^0(\mathbb{P}^n,T_{\mathbb{P}^n}\otimes \mathcal{I}_{Z}(d))$, which proves (\ref{eqinters}).
We thus conclude that
$${\rm dim}\,I_Z^2(2d)\cap J_f^{2d}\leq (n+1){\rm dim}\,I_Z(d+1)$$
which is, for $d$ (hence $n$) large enough,  much smaller than ${\rm dim}\,I''_1$. It follows that, taking representatives
of $k_y,\,\alpha$ in $I_Z^2(d)$, the equation (\ref{equnpeumodifiee11})
provides  an actual vanishing
\begin{eqnarray}\label{equnpeumodifiee1} xa'' k_y=\alpha a''y \,\,{\rm in}\,\,I_Z^2(2d)
\end{eqnarray}

for some  nonzero  $a''\in S^d$, which implies that  $xk_y=\alpha y$ in $I_Z^2(d+1)$. Using the fact that the space
of $(y,k_y)$ satisfying this property has dimension $\geq n+1$, we conclude that $k_y=0$ for generic $(y,k_y)$ and thus $ \alpha=0$.
Proposition \ref{leschiffer} is now proved.
\end{proof}
\begin{rema}\label{remaquonvautiliser}{\rm  Note that, in turn, $\alpha=0$ and equation (\ref{equnpeumodifiee1}) imply that $k_y=0$, so that we also proved that $I_{d-1}^d=x^{d-1}S^1$ mod $J_f^d$. This will be used below.
}
\end{rema}
\subsection{Proof of Theorem \ref{theodonagivoisin}\label{secproofmain}}
We conclude in this section the proof of Theorem \ref{theodonagivoisin}. We start by establishing the following.

\begin{prop}\label{proschigen} Let $f$ be a generic homogeneous  polynomial  of degree $d$ in $n+1$ variables with $d$ dividing $n+1$ and $d$ large enough.
 Let $\phi\in R^d_f$, $I_k^{*d}\subset R^{*d}_f$,  for $*\leq 3$ and $1\leq k\leq d-1$  satisfy  condition (**)  of  section
 \ref{secschideg}.
Then $\phi$ is a (first order) Schiffer variation of $f$.
\end{prop}
\begin{proof} Proposition \ref{leschiffer} proves Proposition \ref{proschigen} when $f=\sum_{j=1}^mf_jg_j$ is the singular polynomial
used in previous section. It thus remains to see that this implies the same result for
the generic $f$. This almost follows because the condition (**) is closed on $(f,\,\phi)$ once the dimensions of the spaces
$R^d_f,\,R^{2d}_f,\,R^{3d}_f$ remain respectively equal to the dimensions of the spaces
$R^d_{f_{\rm gen}},\,R^{2d}_{f_{\rm gen}},\,R^{3d}_{f_{\rm gen}}$ for the generic $f_{\rm gen}$, which is guaranteed for $d$ large enough by Lemma \ref{leconiveausing}.
This is not completely true because we did not prove the statement of Proposition \ref{leschiffer} schematically for the special $f$.
In fact, what we have to do in order to conclude is to prove the following complement to Proposition
\ref{leschiffer}.
\begin{lemm} \label{lecomplement}
 Let the notation and assumption on $f,\,\phi$  be as in Proposition \ref{leschiffer}. Assume moreover that
 $\phi=x^d$ in $R^d_f$, with $x$  generic in $S^1$. Then $I_k^d=x^{k}R^{d-k}_f$ for $1\leq k\leq d-1$.

 Furthermore, inside
 $\mathbb{P}(R^d_f)\times \prod_{k=0}^{d-1}{\rm Grass}(r_k,R^d_f)$, where $r_k:={\rm dim}\,R^{d-k}_f$, the set of points $(\phi=x^d,I_{k}^d=x^kR^{d-k}_f)$,
 is schematically defined, at least at its generic point,  by the condition (*).
\end{lemm}
\begin{rema} {\rm The spaces $I_k^{2d},\,I_k^{3d}$ are defined by the spaces $I_k^d$ using equation (\ref{item3}), so we can consider condition (*) as  a condition on $(\phi,\,I_{k}^d)$ only.}
\end{rema}
\begin{rema}{\rm We did not use up to now equation (\ref{item4}) of Condition (*). We will need it for the proof of this lemma.}
\end{rema}
\begin{proof}[Proof of Lemma \ref{lecomplement}] We already noted in Remark \ref{remaquonvautiliser} that $I_{d-1}^d=x^{d-1}R^{d-1}_f$. We
 study again  the equations
\begin{eqnarray}\label{eqpouraw} aw=x^d\gamma\,\,{\rm in}\,\,R^{2d}_f,
\end{eqnarray}
for $a$ in a subspace $I_1^d\subset R^d_f$    of dimension ${\rm dim}\,R^{d-1}_f$, and $w$ in  $I_{d-1}^d=x^{d-1} S^1$.
We already proved in Lemma \ref{lepourproschIun} that
$\overline{I}_1^d=I^d_{1\mid Z}\subset xH^0(Z,\mathcal{O}_Z(d-1))$.
 We thus conclude that elements  $a\in I_1^d$  can be written as
 $$a=xa'+k_a\,\,{\rm mod}\,\,J_f^d,$$
 with $a'\in S^{d-1},\,k_a\in I_Z(d)$.
 Restricting (\ref{eqpouraw}) to $Z$, we also get
 $$\gamma=a'w'+\gamma'\,\,{\rm mod}\,\,J_f^d,$$
 for some $\gamma'\in I_Z(d)$. The equation (\ref{eqpouraw}) then becomes
 \begin{eqnarray}\label{eqpourawprime} x^{d-1}w'k_a=x^d\gamma'\,\,{\rm in}\,\,R^{2d}_f,
\end{eqnarray}
where $w'$ is generic in $S^1$.
  One then easily concludes that $k_a=0$ mod $\langle xS^{d-1},J_f\rangle$, that is, $I_1^d\subset xR^{d-1}_f$. Hence we proved
  (by dimension reasons) that  \begin{eqnarray}\label{eqIunequalIun} I_1^d=xR^{d-1}_f.
   \end{eqnarray} We now use (\ref{item4}).
  We get  $$ I_1^d\cdot I_1^d\subset I_2^{2d},$$
 which provides, using (\ref{eqIunequalIun})
 $ x^2R^{2d-2}_f\subset I_2^{2d}$.
     Using (\ref{seccond*}), this inclusion   gives  in turn,  by dimension reasons,
    \begin{eqnarray}\label{eqIunIun}x^2R^{2d-2}_f= I_2^{2d}.
    \end{eqnarray}
    Here, in order to apply the dimension argument,  we need to know  that multiplication by $x^2$
    is injective on $R^{2d-2}_f$. More generally we will need to know that  multiplication by
$x^{i}$ is injective  on $R^{d+i}_f$ for $1\leq i\leq d$, which is not hard to prove since $x$ is generic.
    We next use (\ref{item3})
     $$R^d_fI_2^{d}\subset I_2^{2d},$$
   that is, $$I_2^d\subset [I_2^{2d}:R^d_f]=[x^2R^{2d-2}_f:R^d_f],$$ and easily conclude that $I_2^d\subset x^2R^{d-2}_f$, hence
   $I_2^d=  x^2R^{d-2}_f$ by dimension reasons. We continue this way and prove that
   $I_{k}^d=x^kR^{d-k}_f$ for all $1\leq k\leq d-1$.
   Thus the first statement is proved.

In order to prove the schematic statement, we consider a first order variation $(h,h_1,\ldots,h_{d-1})$ of $(x^d,I^d_1,\ldots,I^d_{d-1})$ satisfying conditions
(*) at first order. We thus have a  first order deformation
$x^d+\epsilon h\in R^d_f$ of $x^d$ and
$$h_1\in {\rm Hom}\,(I^d_1,R^d_f/I^d_1),\,\,\ldots,\,h_{d-1}\in {\rm Hom}\,(I^d_{d-1},R^d_f/I^d_{d-1}),$$
satisfying the infinitesimal version of the equations (\ref{item2})-(\ref{item4}).
We have to prove that there is a $y\in S^1/x$ such that
$$h_l:I^d_1\cong R^{d-l}_f\rightarrow R^d_f/x^lR^d_f$$
is given by multiplication by $lyx^{l-1}$.
We first observe that it suffices to prove the result for $l=1$, because the reasoning above, which deduces the equality
$I_{k}^d=x^kR^{d-k}_f$ for all $1\leq k\leq d-1$ from the equality (\ref{eqIunequalIun}) using equations (\ref{item3}) and (\ref{item4}) work as well schematically.

We thus have a  first order deformation
$x^d+\epsilon h\in R^d_f$ of $x^d$ and
$$h_1\in {\rm Hom}\,(I^d_1,R^d_f/I^d_1),\,\,h_{d-1}\in {\rm Hom}\,(I^d_{d-1},R^d_f/I^d_{d-1}),$$
satisfying the equations
\begin{eqnarray}\label{eqinfiversion}(a+\epsilon h_1(a))(w+\epsilon h_{d-1}(w))=(x^d+\epsilon h)\gamma_\epsilon\,\,{\rm in}\,\,R^{2d}_f\otimes \mathbb{C}[\epsilon]/(\epsilon^2),
\end{eqnarray}
for any $a=x a'\in xR^{d-1}_f$, $w=x^{d-1}w'\in x^{d-1}S^1$, and for $\gamma_\epsilon=\gamma+\epsilon\gamma_1$, where $\gamma $
is as in (\ref{eqpouraw}). We want to prove that there exists $y\in S^1/\langle x\rangle$, such that for any $a=xa'$, the following holds in $R^d_f$
\begin{eqnarray}\label{eqpourh1finfin}  h_1(a)=ya'\,\,{\rm mod} \,\,xR^{d-1}_f.
\end{eqnarray}

Looking at the previous proof, we deduce from (\ref{eqpourawprime}) with $k_a=0$ that $\gamma'=0$ (using injectivity of the multiplication by $x^d$), so  $\gamma=a'w'$ in $R^d_f$. We thus have $\gamma_\epsilon=a'w'+\epsilon \gamma_1$.
Equation (\ref{eqinfiversion}) then gives
\begin{eqnarray}\label{eqinfiversiondev} h_1(a)x^{d-1}w'+ h_{d-1}(w)xa'=x^d\gamma_1+ ha'w' \,\,{\rm in}\,\,R^{2d}_f,
\end{eqnarray}
for any $a=x a'\in xR^{d-1}_f$, $w=x^{d-1}w'\in x^{d-1}S^1$.
We claim   that
\begin{eqnarray}\label{eqinfinfinfin}h_1(a)_{\mid Z}\in \langle a'\rangle\,\,{\rm mod}\,\,\langle x\rangle.
\end{eqnarray}
Indeed,  (\ref{eqinfiversiondev})  first implies that $h=xh'$ since it becomes divisible by $x$ after multiplication by any element of
$R^d_f$, and then, after simplification by $x$, that
\begin{eqnarray}\label{eqinfiversiondevrestZ} h_1(a)x^{d-2}w'+ h_{d-1}(w)a'=x^{d-1}\gamma_1+ h'a'w' \,\,{\rm in}\,\,R^{2d-1}_f.
\end{eqnarray}
We rewrite (\ref{eqinfiversiondevrestZ}) in the form
\begin{eqnarray}\label{eqinfiversiondevrestZtripot} x^{d-2}(h_1(a)w'-x\gamma_1)+a'(h_{d-1}(w)-h' w')=0.
\end{eqnarray}
We now restrict (\ref{eqinfiversiondevrestZtripot}) to $Z$. As $x^{d-2}$ and $a'$ have no common divisor on $Z$ for $a'$ generic, it follows that
$$(h_1(a)w'-x\gamma_1)_{\mid Z}\in\langle a'\rangle,$$
which proves (\ref{eqinfinfinfin}) since $w'\in S^1$ is generic.

We can even conclude by similar arguments that
$$ h_1(a)_{\mid Z}= m_{1} a' \,\,\,{\rm mod}\,\,\langle x\rangle,$$
for some $m_1\in H^0(Z,\mathcal{O}_Z(1))$.
We can see $m_1$ as an  element  $y\in S^1$ because  the map of restriction  to $Z$ is an isomorphism in degree $1$, and we can thus write in $R_f^d$
\begin{eqnarray}\label{eqfinfuin} h_1(a)=y a'+k_1(a')\,\,\,{\rm in}\,\,R^d_f/x R^{d-1}_f,
\end{eqnarray}
where $k_1(a')\in I_Z(d)$ for any $a'\in R^{d-1}_f$.
Equation (\ref{eqinfiversiondevrestZtripot}) then gives $x^{d-2}((y a'+k_1(a'))w'-x\gamma_1)+a'(h_{d-1}(w)-h' w')=0$ in $R^{2d-1}_f$, that is
\begin{eqnarray}\label{eqregroupfin} x^{d-2}(k_1(a')w'-x\gamma_1)+a'(y x^{d-2}h_{d-1}(w)-h' w')=0\,\,{\rm in}\,\,R^{2d-1}_f.
\end{eqnarray}
The term $x^{d-2}(k_1(a')w'-x\gamma_1)$ belongs by (\ref{eqregroupfin}) to $x^{d-2}I_Z(d+1)\cap \langle a'\rangle$. For $a'$ generic, it is easy to  show that it implies that it  belongs to $a'x^{d-2}I_Z(2)=0$. Thus $x^{d-2}(k_1(a')w'-x\gamma_1)=0\,\,{\rm in}\,\,R^{2d-1}_f$, hence $k_1(a')w'-x\gamma_1=0$, and $k_1(a')=0$ mod $\langle x\rangle$. This is true for $a'$ generic in $R^{d-1}_f$, hence for all $a'$. Thus (\ref{eqpourh1finfin}) is proved.
\end{proof}

Lemma \ref{leschiffer} is a schematic version of Proposition \ref{leschiffer} that guarantees that the Veronese
image $v_f(\mathbb{P}(S^1))\subset \mathbb{P}(R^d_f)$ is characterized not only set theoretically but also schematically (at the generic point) by condition (**) (in fact, we can see from the proof above that  condition (*) even suffices for the scheme-theoretic statement, but condition (**) was needed to prove the set-theoretic statement for the special $f$). It  follows that
for generic $f_{\rm gen}$,  the Veronese
image $v_f(\mathbb{P}(S^1))\subset \mathbb{P}(R^d_f)$  is also  characterized by condition (**).
\end{proof}
\begin{proof}[Proof of Theorem \ref{theodonagivoisin} (1)]
Fix integers $d,\,n$ with $d$ dividing $n+1$, and for which the conclusion  of Proposition
\ref{proschigen} holds. We want to show that if $X_f$ is  a   very general hypersurface of degree $d$ in $\mathbb{P}^n$, then any smooth hypersurface $X_g$ of degree $d$ in $\mathbb{P}^n$
such
that there exists an isomorphism
$$ H^{n-1}(X_g,\mathbb{Q})_{\rm prim}\cong H^{n-1}(X_f,\mathbb{Q})_{\rm prim}$$
of rational Hodge structures, is isomorphic to $X_f$.

We first argue as in Section \ref{secnounou}. Denote by $U_{d,n}^0\subset U_{d,n}$ the Zariski open set parametrizing automorphisms free smooth hypersurfaces. As $f$ is very general,  $f\in U_{d,n}^0$ and our assumption provides
simply connected Euclidean open neighborhoods
$U\subset U_{d,n}^0,\,V\subset U_{d,n}^0$ of $f$, $g$ respectively, a holomorphic
 diffeomorphism $i:U\cong V$ with $i(f)=g$, and an isomorphism of complex variations of Hodge structures
 $$ (H^{n-1}_\mathbb{C}, F^\cdot\mathcal{H}^{n-1})\cong i^{-1}(H^{n-1}_\mathbb{C}, F^\cdot\mathcal{H}^{n-1})$$
 on $U$. Here, if $\pi:\mathcal{X}_{d,n}\rightarrow U_{d,n}^0$ is the universal hypersurface,
 $H^{n-1}_\mathbb{C}$ is the local system $R^{n-1}\pi_*\mathbb{C}_{\rm prim}$ on $U_{d,n}^0$,
 and  $F^\cdot\mathcal{H}^{n-1}$ is the Hodge filtration on the associated flat holomorphic vector bundle
 $\mathcal{H}^{n-1}=H^{n-1}_\mathbb{C}\otimes \mathcal{O}_{U_{d,n}^0}$.

 The differential $i_*:T_{U,f}\rightarrow T_{V,g}$ is a linear isomorphism
 $$i_*:R^d_f\cong R^d_g.$$

 \begin{claim}\label{propvero} In the situation described above, the differential $i_*$ sends the set of first order Schiffer  variations of $f$ to the set of first order Schiffer variations of $g$.
 \end{claim}
\begin{proof}  Indeed, the local diffeomeorphism $i$ induces an isomorphism of variations of Hodge structures. It thus sends a $1$-parameter Schiffer variation
$(f_t)_{t\in\Delta}$ of $f$ to a $1$-parameter  variation
$(g_t)_{t\in\Delta}$, $g_t:=i(f_t)$,  of $g$, which satisfies the  assumptions of Proposition \ref{proschigen}.
Proposition \ref{proschigen} then tells us that  $\psi:=\frac{\partial g_t}{\partial t}_{\mid t=0}$ is a first order Schiffer variation of $g$. But  $\phi:=\frac{\partial f_t}{\partial t}_{\mid t=0}$  is an arbitrary first order  Schiffer variation of $f$ and  we have $\psi=i_*(\phi)$.
\end{proof}
Having the claim, the proof of the  theorem is finished  using Proposition \ref{proschifferrec}.
\end{proof}

CNRS, IMJ-PRG, 4 Place Jussieu 75005 Paris, France

claire.voisin@imj-prg.fr
    \end{document}